\documentclass[11pt]{amsart}
\usepackage[margin=1in]{geometry}

\usepackage{amssymb}
\usepackage{amsthm}
\usepackage{amsmath}
\usepackage{mathrsfs}
\usepackage{amsbsy}
\usepackage[all]{xy}
\usepackage{bm}
\usepackage{hyperref}
\usepackage{tikz}
\usepackage{array}
\usepackage{enumerate}
\usepackage{xcolor}
\usepackage{bbm}
\usepackage{comment}
\usepackage{dynkin-diagrams}
\usepackage{ifthen}
\usepackage{thmtools, thm-restate}

\usepackage[noabbrev,capitalise,nameinlink]{cleveref}

\definecolor{lavender}{rgb}{0.4,0,1.0}

\hypersetup{colorlinks=true, citecolor=lavender, linkcolor=lavender,urlcolor=lavender}

\usepackage{hhline}
\allowdisplaybreaks
\usepackage[noadjust]{cite}
\usepackage{asymptote}

\usepackage{caption}
\usepackage{tabu}
\usepackage{diagbox}
\usepackage[noabbrev,capitalise,nameinlink]{cleveref}
\crefname{conjecture}{Conjecture}{Conjectures}

\newtheorem{theorem}{Theorem}[section]

\newtheorem{corollary}[theorem]{Corollary}

\newtheorem{question}[theorem]{Question}

\newtheorem{lemma}[theorem]{Lemma}

\theoremstyle{definition}

\newtheorem{remark}[theorem]{Remark}
\newtheorem{example}[theorem]{Example}

\definecolor{Teal}{rgb}{0,0.784,0.784}
\definecolor{Purple}{rgb}{0.847,0.6,1}
\definecolor{Yellow}{rgb}{0.9,0.9,0.5}

\usepackage{etoolbox}

\newcommand{\includeSymbol}[1]{\ensuremath{%
	\mathchoice
		{\raisebox{-.7mm}{\includegraphics[height=2.2ex]{#1}}}	
		{\raisebox{-.7mm}{\includegraphics[height=2.2ex]{#1}}}
		{\raisebox{-.6mm}{\includegraphics[height=1.6ex]{#1}}}
		{\raisebox{-.5mm}{\includegraphics[height=1ex]{#1}}}
}}

\robustify{\includeSymbol}
\newcommand{\Reflect}{\includeSymbol{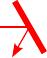}}

\robustify{\includeSymbol}
\newcommand{\Refract}{\includeSymbol{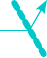}}

\newcommand{\cyc}{\mathrm{cyc}}

\newcommand{\BB}{\mathbb{B}}
\newcommand{\affS}{\widetilde{\mathfrak{S}}}
\newcommand{\HHH}{\mathbf{H}}
\newcommand{\TT}{\mathbb{T}}
\newcommand{\HH}{\mathrm{H}}
\newcommand{\PP}{\mathrm{P}}
\newcommand{\UU}{\mathrm{U}}

\newcommand{\vv}{\mathbf{v}}
\newcommand{\SD}{\mathrm{SD}}

\newcommand{\SYT}{\mathrm{SYT}}
\newcommand{\maj}{\mathrm{maj}}
\newcommand{\ff}{f}
\newcommand{\FF}{F}
\newcommand{\FFF}{\mathbf{F}}

\newcommand{\bb}{\mathfrak{b}}
\newcommand{\s}{\widetilde{s}}
\newcommand{\quot}{\mathbf{q}}

\newcommand{\TPro}{\mathrm{TPro}}
\newcommand{\Z}{\mathbb{Z}}
\newcommand{\XX}{\Xi}
\newcommand{\Cycle}{\mathsf{Cycle}}
\newcommand{\Eflect}{{\color{red}E_{\Reflect}}}
\newcommand{\Efract}{{\color{Teal}E_{\Refract}}}
\newcommand{\Wflect}{{\color{red}\mathcal{W}_{\Reflect}}}
\newcommand{\Wfract}{{\color{Teal}\mathcal{W}_{\Refract}}}
\newcommand{\WW}{\mathcal{W}}

\newcommand{\dfn}[1]{\textcolor{blue}{\emph{#1}}}

\begin{document}

\title[]{Toric Promotion with Reflections and Refractions}
\subjclass[2010]{}

\author[]{Ashleigh Adams}
\address[]{Department of Mathematics, North Dakota State University, Fargo, ND 58102, USA}
\email{ashleigh.adams@ndsu.com}

\author[]{Colin Defant}
\address[]{Department of Mathematics, Harvard University, Cambridge, MA 02138, USA}
\email{colindefant@gmail.com}

\author[]{Jessica Striker}
\address[]{Department of Mathematics, North Dakota State University, Fargo, ND 58102, USA}
\email{jessica.striker@ndsu.com}

\maketitle

\begin{abstract}
Inspired by recent work on refraction billiards in dynamics, we introduce a notion of refraction for combinatorial billiards. This allows us to define a generalization of toric promotion that we call \emph{toric promotion with reflections and refractions}, which is a dynamical system defined via a graph $G$ whose edges are partitioned into a set of \emph{reflection edges} and a set of \emph{refraction edges}. This system is a discretization of a billiards system in which a beam of light can pass through, reflect off of, or refract through each toric hyperplane in a toric arrangement. Generalizing the main theorem known about toric promotion, we give a simple formula for the orbit structure of toric promotion with reflections and refractions when $G$ is a forest. We also completely describe the orbit sizes when $G$ is a cycle with an even number of refraction edges; this result is new even for ordinary toric promotion (i.e., when there are no refraction edges). When $G$ is a cycle of even size with no reflection edges, we obtain an interesting instance of the cyclic sieving phenomenon. 
\end{abstract} 

\section{Introduction}\label{sec:intro}

\subsection{Combinatorial Billiards}

The field of \emph{dynamical algebraic combinatorics} concerns objects of interest in combinatorics, drawing inspiration from classical dynamical systems for guidance on the relevant questions and themes worth investigating. Most of the work in this area has focused on dynamical systems obtained via the iteration of certain combinatorially defined operators; two of the most prominent such operators are \emph{promotion} and \emph{rowmotion} \cite{BSV,Semidistrim,DHPP,Haiman,HopkinsRubey,ProppRoby,Rhoades,Schutzenberger1,Schutzenberger2,StanleyPromotion,StrikerSurvey,StrikerWilliams,ThomasWilliams}. A new subfield of dynamical algebraic combinatorics, which we call \emph{combinatorial billiards}, draws inspiration from a different part of classical dynamics: mathematical billiards. Roughly speaking, combinatorial billiards studies discretized versions of billiards systems that are much more rigid than usual billiards systems. The articles \cite{DefantJiradilok,Zhu} prove extremal combinatorial results about billiards systems that take place inside of polygons in the equilateral triangular grid; such billiards systems can also be interpreted in terms of trip permutations of certain plabic graphs. The article \cite{Barkley} relates combinatorial billiards systems to tiling enumeration problems. The article \cite{BDHKL} studies Coxeter-theoretic \emph{Bender--Knuth billiards systems}, in which a discretized beam of light bounces around in an arrangement of transparent windows and one-way mirrors. 

\begin{figure}[ht]
  \begin{center}
  \includegraphics[height=5.1cm]{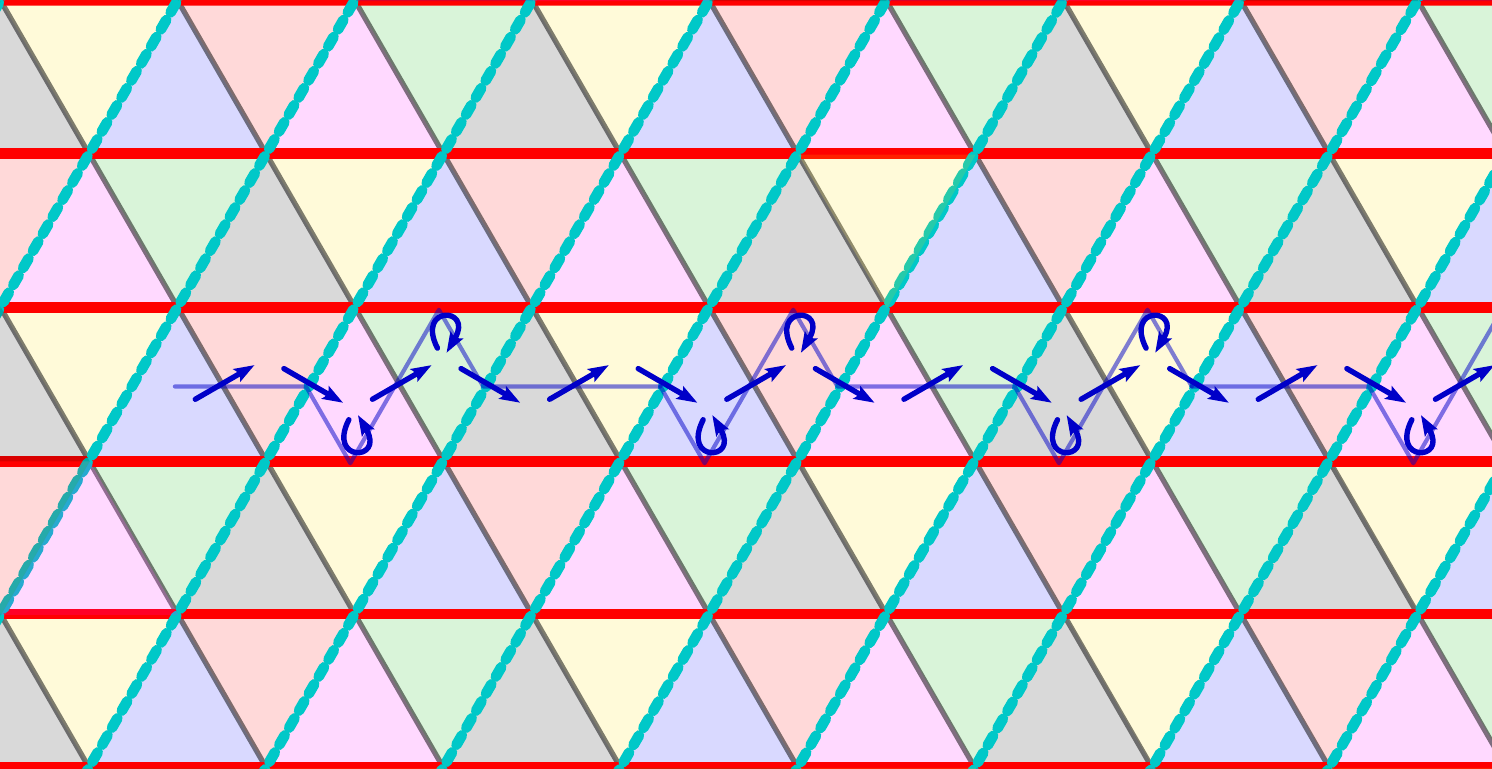}
  \quad 
  \includegraphics[height=5.125cm]{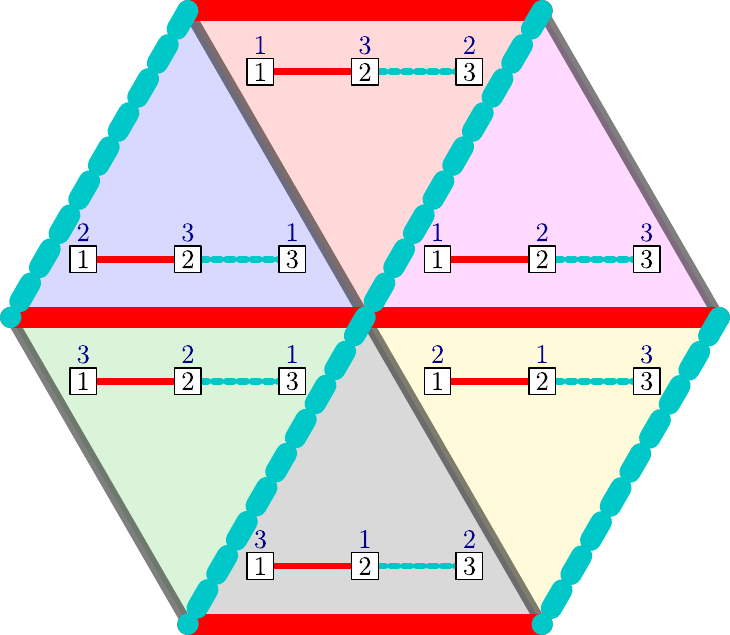} \\ 
  \vspace{0.5cm}
  \includegraphics[width=\linewidth]{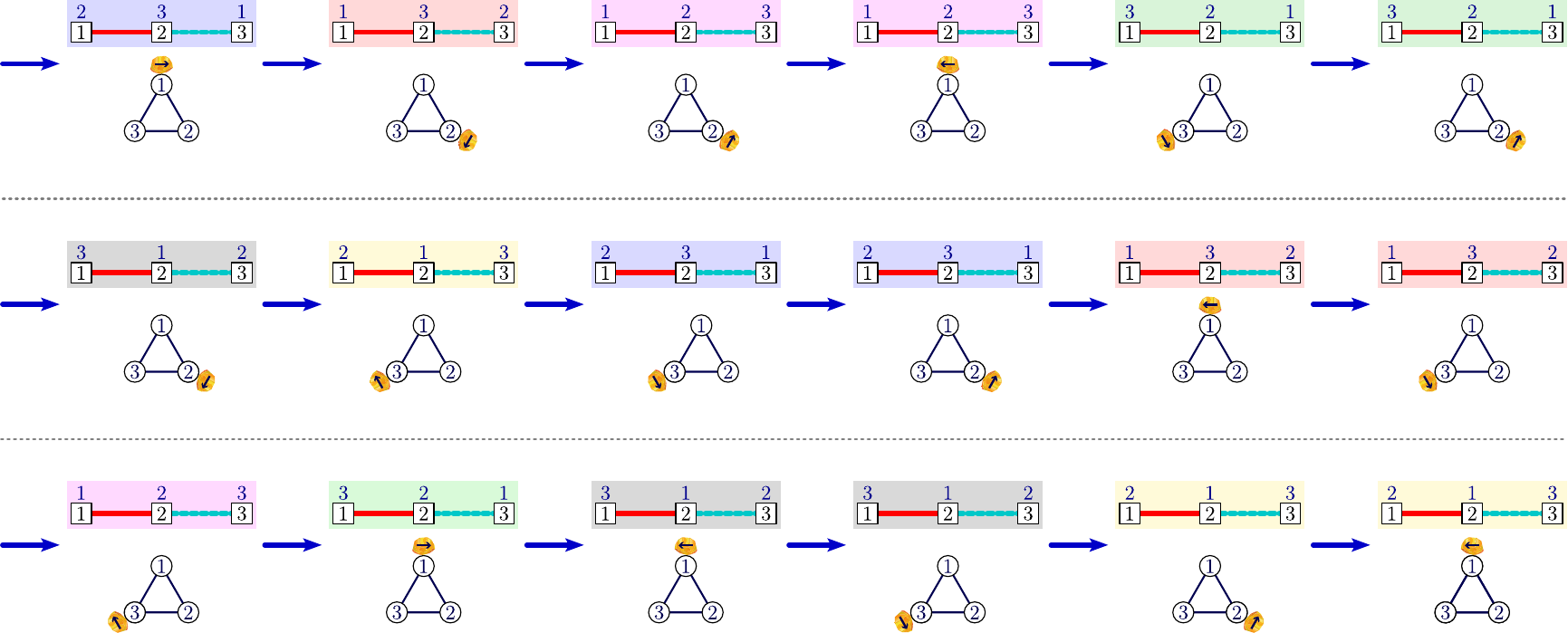}
  \end{center}
\caption{The image in the upper left shows an orbit of a combinatorial billiards system with reflections and refractions in the equilateral triangular grid. Each of the lines is a {\color{gray}grey} \emph{window}, a {\color{red}red} \emph{mirror}, or a {\color{Teal}teal} \emph{metalens}. Unit triangles correspond to elements of the affine symmetric group $\affS_3$. Quotienting by the action of the coroot lattice, we obtain the toric arrangement shown in the upper right. Toric regions of this toric arrangement correspond to permutations in the symmetric group $\mathfrak S_3$. Projecting the orbit of the combinatorial billiards system to the torus yields a size-$18$ orbit of toric promotion with reflections and refractions, which is shown in the bottom image. }\label{fig:start}
\end{figure}

In \cite{DefantToric}, Defant introduced a new combinatorial dynamical system called \emph{toric promotion}, which acts on labelings of a graph and serves as a cyclic analogue of Sch\"utzenberger's famous promotion operator. The original work on toric promotion was purely combinatorial. In this article, we give a geometric description of toric promotion, thereby placing it into the realm of combinatorial billiards. 

An exciting variant of classical billiards that has come under consideration in recent years is \emph{refraction billiards}, in which a beam of light can refract (i.e., bend) instead of reflecting \cite{Baird,Barutello,Davis1,Davis2,DeBlasi2,DeBlasi1,Jay,Paris}. The primary novel idea introduced in this article is a definition of refraction in combinatorial billiards. This new framework will allow us to introduce and investigate a generalization of toric promotion that uses reflections and refractions. We were originally led to the definition of a refraction when attempting to model trip permutations of certain \emph{hourglass plabic graphs} (as introduced in \cite{GPPSS}) in the same way that triangular-grid billiards systems were used in \cite{DefantJiradilok} to model trip permutations of certain ordinary plabic graphs. 

For now, let us define toric promotion with reflections and refractions in purely combinatorial terms, postponing the geometric formulation until \cref{sec:billiards}. 

Fix an integer $n\geq 3$. For $i\in\Z/n\Z$, let $s_i=(i\,\,i+1)$ denote the transposition that swaps $i$ and $i+1$ (in particular, $s_n$ swaps $n$ and $1$). Let $\Cycle_n$ denote the cycle graph with vertex set $\Z/n\Z$ and with edge set $\{\{i,i+1\}:i\in\Z/n\Z\}$. Fix an embedding of $\Cycle_n$ in the plane so that the vertices are arranged in the clockwise cyclic order $1,2,\ldots,n$. 

Let $G=(V,E)$ be a simple graph with $n\geq 3$ vertices. A \dfn{labeling} of $G$ is a bijection $V\to\Z/n\Z$. Let $\Lambda_G$ denote the set of labelings of $G$. Let $E=\Eflect\sqcup\Efract$ be a partition of the edge set of $G$ into two sets $\Eflect$ and $\Efract$. Elements of $\Eflect$ are called \dfn{reflection edges}, while elements of $\Efract$ are called \dfn{refraction edges}.\footnote{We suggest pronouncing the symbols $\Eflect$ and $\Efract$ as ``E-flect'' and ``E-fract,'' respectively.} 

Let $\XX_G=\Lambda_G\times\Z/n\Z\times\{\pm 1\}$. Define a map $\Theta\colon\XX_G\to\XX_G$ by 
\begin{equation}\label{eq:Theta}
\Theta(\sigma,i,\epsilon)=\begin{cases} (s_i\circ\sigma,i+\epsilon,\epsilon) & \mbox{if }\{\sigma^{-1}(i),\sigma^{-1}(i+1)\}\not\in E; \\   (\sigma,i+\epsilon,\epsilon) & \mbox{if }\{\sigma^{-1}(i),\sigma^{-1}(i+1)\}\in \Eflect; \\   (s_i\circ\sigma,i-\epsilon,-\epsilon) & \mbox{if }\{\sigma^{-1}(i),\sigma^{-1}(i+1)\}\in \Efract.  \end{cases}
\end{equation} 
The dynamical system on $\XX_G$ determined by $\Theta$, which we call \dfn{toric promotion with reflections and refractions}, is our central object of study. 

We imagine that the elements of $\{\pm 1\}$ represent the two cyclic orientations of $\Cycle_n$: $1$ represents clockwise, while $-1$ represents counterclockwise. Thus, a triple $(\sigma,i,\epsilon)\in\XX_G$ encodes a labeling of $G$, a vertex of $\Cycle_n$, and a cyclic orientation of $\Cycle_n$. It is helpful to imagine that there is a stone placed on the vertex $i+\frac{1}{2}(1-\epsilon)$ of $\Cycle_n$ and that the stone is pointing toward $i+\frac{1}{2}(1+\epsilon)$. When we apply $\Theta$ to $(\sigma,i,\epsilon)$, there are three possibilities for what can happen. If $\sigma$ assigns the labels $i$ and $i+1$ to nonadjacent vertices of $G$, then we swap the labels $i$ and $i+1$ and move the stone one step in the direction it is pointing. If $\sigma$ assigns the labels $i$ and $i+1$ to the vertices of a reflection edge, then we do nothing to the labeling and move the stone one step in the direction it is pointing. If $\sigma$ assigns the labels $i$ and $i+1$ to the vertices of a refraction edge, then we swap the labels $i$ and $i+1$ and reverse the direction of the stone. 

\begin{example}
The bottom image in \cref{fig:start} shows an orbit of $\Theta$ of size $18$, where \[G=\begin{array}{l}\includegraphics[height=0.353cm]{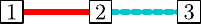}\end{array}\] is the path graph with vertex set $\{1,2,3\}$ with a single reflection edge $\{1,2\}$ and a single refraction edge $\{2,3\}$. We draw a triple $(\sigma,i,\epsilon)$ with the labeling $\sigma$ above a diagram representing the pair $(i,\epsilon)$. For example, 
\[\begin{array}{l}\includegraphics[height=2.5cm]{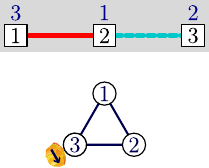}\end{array}\]
represents the triple $(\sigma,2,-1)$, where $\sigma(1)=3$, $\sigma(2)=1$, and $\sigma(3)=2$. 
We will explain the geometric parts of \cref{fig:start} in \cref{sec:affine_S_n,sec:billiards}. 
\end{example} 

The toric promotion operator $\TPro\colon\Lambda_G\to\Lambda_G$ introduced in \cite{DefantToric} appears in our setting when there are no refraction edges. More precisely, if $\Efract=\varnothing$ (so $E=\Eflect$), then $\TPro(\sigma)$ is the labeling (i.e., the first entry) in the triple $\Theta^n(\sigma,1,1)$. 

As we will discuss in \cref{sec:billiards}, iteration of $\Theta$ can be interpreted as a discretized billiards system in which a beam of light travels in a torus, sometimes reflecting or refracting when it hits the walls of a toric hyperplane arrangement. In a triple $(\sigma,i,\epsilon)\in\XX_G$, the labeling $\sigma$ encodes the location of the beam of light, while the index $i$ and the orientation $\epsilon$ together encode the direction that the beam of light is facing. A size-$k$ orbit of $\Theta$ corresponds to a periodic discrete billiards trajectory of period~$k$. Thus, our main theorems concern the orbit structure of $\Theta$. For $i\in\Z/n\Z$ and $\epsilon\in\{\pm 1\}$, the map $\boldsymbol{\omega}_{i,\epsilon}\colon\Z/n\Z\to\Z/n\Z$ defined by $\boldsymbol{\omega}_{i,\epsilon}(j)=\epsilon(j-i)+1$ is an automorphism of $\Cycle_n$, and the orbit of $\Theta$ containing $(\sigma,i,\epsilon)$ has the same size as the orbit of $\Theta$ containing $(\boldsymbol{\omega}_{i,\epsilon}\circ\sigma,1,1)$. Therefore, when aiming to compute the size of the orbit of $\Theta$ containing $(\sigma,i,\epsilon)$, we may assume without loss of generality that $i=1$ and $\epsilon=1$.

\subsection{Further Motivation}\label{subsec:remarks}

We believe that the connection we draw between dynamical algebraic combinatorics and mathematical billiards will be of interest to researchers in both fields. On the one hand, there are several fascinating refraction billiards systems studied in the articles \cite{Baird,Barutello,Davis1,Davis2,DeBlasi2,DeBlasi1,Jay,Paris}. As illustrated by our main results, the types of statements that one can prove about combinatorial billiards are much more precise. Our algebraic/combinatorial framework also allows us to work in high-dimensional spaces; this is in contrast to the aformentioned articles, which often work in $2$-dimensional regimes. It seems potentially fruitful to ``discretize'' some of the systems from those articles in order to obtain more precise statements and to generalize them to higher dimensions. On the other hand, our model provides several new examples of interesting combinatorial dynamical systems for researchers in the combinatorics community to explore. We expect that there are other graphs beyond forests and cycles where one could prove interesting results about toric promotion with reflections and refractions. 

In addition to the aforementioned connection between combinatorics and dynamics, there is also an interesting topological interpretation of our work, suggested to us by Pavel Galashin. Given an $n$-vertex graph $G=(V,E)$ and a partition $E=\Eflect\sqcup\Efract$, one can consider an orbit of $\Theta$ as a periodic billiards trajectory in an $(n-1)$-dimensional torus. By tracing out one period of this billiards trajectory, one obtains a closed loop in the torus. We believe it would be very interesting to understand these loops topologically. In particular, one can ask when they are contractible (see \cref{quest:contractible}). 

Another advantage of our geometric perspective in terms of billiards is that it naturally leads to other combinatorial dynamical systems that we have not yet attempted to study but that we believe should yield interesting properties. We will see in \cref{sec:billiards} that our billiards system is obtained by shining a beam of light in the direction of a very specific vector in the type-$A_{n-1}$ coroot lattice. By choosing a different coroot vector as the initial direction, one would obtain a different combinatorial dynamical system with its own potentially-interesting properties. As mentioned in \cref{subsec:Weyl}, one could obtain even more systems by considering other affine Weyl groups.

\subsection{Forests}

The main result in Defant's original article about toric promotion (see \cite[Theorem~1.3]{DefantToric}) gives an exact description of the orbit structure of toric promotion when the underlying graph $G$ is a forest (i.e., a graph with no cycles) with no refraction edges. Our first main theorem vastly generalizes this result by allowing an arbitrary mixture of reflection edges and refraction edges. In order to state it, we need a bit more terminology. 

Let $G=(V,E)$ be a forest, and let $E=\Eflect\sqcup\Efract$ be a partition of its edge set. Let $T=(V_T,E_T)$ be a connected component of $G$. We can partition the vertex set of $T$ into two (disjoint) subsets $X_1$ and $X_{-1}$ so that 
\begin{itemize}
\item every edge in $\Efract\cap E_T$ has one endpoint in $X_1$ and one endpoint in $X_{-1}$; 
\item every edge in $\Eflect\cap E_T$ either has both its endpoints in $X_1$ or has both its endpoints in $X_{-1}$.
\end{itemize}
This partition is unique up to swapping the roles of $X_1$ and $X_{-1}$, so we may define \[\chi(T;\Eflect,\Efract)=||X_1|-|X_{-1}||.\]  

\begin{theorem}\label{thm:forest}
Let $G=(V,E)$ be a forest with $n\geq 3$ vertices, and fix a partition $E=\Eflect\sqcup\Efract$ of $E$ into a set of reflection edges and a set of refraction edges. Let $(\sigma,i,\epsilon)\in\XX_G$. Let $T=(V_T,E_T)$ be the connected component of $G$ containing the vertex $\sigma^{-1}(i+\frac{1}{2}(1-\epsilon))$. The orbit of $\Theta$ containing $(\sigma,i,\epsilon)$ has size \[\frac{|V_T|n(n-1)}{\gcd(n,\chi(T;\Eflect,\Efract))}.\]
\end{theorem}

\begin{remark}
Suppose $G$ is an $n$-vertex forest whose edges are all reflection edges. For every connected component $T$ of $G$, we have ${\chi(T;\Eflect,\Efract)=\chi(T;\Eflect,\varnothing)=|V_T|}$, so \cref{thm:forest} states that the orbit of $\Theta$ containing a triple $(\sigma,i,\epsilon)\in\Xi_G$ has size \[\frac{|V_T|n(n-1)}{\gcd(n,|V_T|)}.\] This recovers \cite[Theorem~1.3]{DefantToric}. 
\end{remark}

\begin{corollary}\label{cor:tree}
Let $G=(V,E)$ be a tree with $n\geq 3$ vertices, and fix a partition $E=\Eflect\sqcup\Efract$ of $E$ into a set of reflection edges and a set of refraction edges. All orbits of $\Theta$ have size \[\frac{n^2(n-1)}{\gcd(n,\chi(G;\Eflect,\Efract))}.\]
\end{corollary}

To illustrate the preceding corollary, suppose $G$ is an $n$-vertex path graph whose edges are all refraction edges. Then \[\chi(G;\Eflect,\Efract)=\chi(G;\varnothing,E)=\begin{cases} 1 & \mbox{if }n\text{ is odd}; \\   0 & \mbox{if }n\text{ is even}. \end{cases}\] According to \cref{cor:tree}, all orbits of $\Theta$ have the same size; this size is $n^2(n-1)$ if $n$ is odd and is $n(n-1)$ if $n$ is even. 

For another illustration, suppose $G$ is an $n$-vertex star graph (i.e., a tree with a single vertex of degree $n-1$ and $n-1$ vertices of degree $1$). Then $\chi(G;\Eflect,\Efract)=|n-2|\Efract||$, so it follows from \cref{cor:tree} that all orbits of $\Theta$ have size \[\frac{n^2(n-1)}{\gcd(n,2|\Efract|)}.\]

\begin{remark}
The methods we use to prove \cref{thm:forest} are reminiscent of the those used to study \emph{tree-like factorizations} in \cite{NathanCayley}. However, that article only deals with reflections (not refractions). 
\end{remark}

\subsection{Cycles}

Assume $G=(V,E)$ is a cycle graph, and let $E=\Eflect\sqcup\Efract$ as before. Suppose in addition that $|\Efract|$ is even. Our aim is to find the size of the orbit of $\Theta$ containing a triple $(\sigma,i,\epsilon)\in\XX_G$. As mentioned above, we may assume without loss of generality that $i=1$ and $\epsilon=1$ (this is just to ease the exposition). 

Let us embed $G$ in the plane and name its vertices in the clockwise cyclic order as $v_1,v_2,\ldots,v_n$. By choosing how we embed $G$ in the plane and how we choose $v_1$, we may assume that $\sigma(v_n)=1$ and that $\sigma(v_1)<\sigma(v_{n-1})$ when we identify $\Z/n\Z$ with $[n]$ in the obvious manner. For each vertex $v_k\in V$, we consider a formal symbol $\vv_k$ that we call the \dfn{replica} of $v_k$. Let us place the replica $\vv_k$ on the vertex $\sigma(v_k)$ of $\Cycle_n$. If we start at the replica $\vv_{n-1}$, we can walk clockwise along $\Cycle_n$ until reaching $\vv_1$; let $a_0$ be $1$ more than the number of replicas other than $\vv_n$ (not including $\vv_{n-1}$ or $\vv_1$) through which we cross during this walk. Similarly, if we start at $\vv_1$, we can walk clockwise until reaching $\vv_2$; let $a_1$ be $1$ more than the number of replicas other than $\vv_n$ that we cross during this walk. Starting at $\vv_2$, we can walk clockwise until reaching $\vv_3$; let $a_2$ be $1$ more than the number of replicas other than $\vv_n$ that we cross during this walk. Repeating this process, we obtain numbers $a_0,a_1,\ldots,a_{n-2}$. We can then form an infinite sequence $(a_\beta)_{\beta\geq 0}$ by declaring that $a_{j+n-1}=a_j$ for all $j\geq 0$. It is a simple consequence of these definitions that $\sum_{\ell=0}^{n-2}a_\ell=m_\sigma(n-1)$ for some positive integer $m_\sigma$. Let $p_\sigma$ be the period of the sequence $(a_\beta)_{\beta\geq 0}$, and note that $p_\sigma$ divides $n-1$. 

Because $G$ has an even number of refraction edges, there is a unique partition $V=Y_1\sqcup Y_{-1}$ of the vertex set of $G$ such that 
\begin{itemize}
\item every refraction edge has one endpoint in $Y_1$ and one endpoint in~$Y_{-1}$; 
\item every reflection edge either has both its endpoints in $Y_1$ or has both its endpoints in~$Y_{-1}$; 
\item $v_n\in Y_1$.
\end{itemize}
Define $\mu_\sigma=|Y_1|$.

\begin{example}\label{exam:cycle1}
Assume $n=7$, and let $\sigma$ be the labeling of the cycle graph $G$ satisfying \[\sigma(v_1)=5,\quad\sigma(v_2)=6,\quad\sigma(v_3)=4,\quad\sigma(v_4)=2,\quad\sigma(v_5)=3,\quad\sigma(v_6)=7,\quad\sigma(v_7)=1.\] Note that, as stipulated above, we have $\sigma(v_n)=1$ and $\sigma(v_1)<\sigma(v_{n-1})$ (that is, $5<7$). We can represent the triple $(\sigma,1,1)$ and the graph $G$ as follows: 
\[\begin{array}{l}\includegraphics[height=3.242cm]{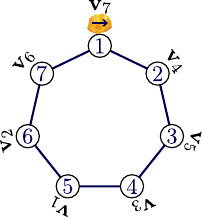}\end{array}\qquad\qquad\raisebox{-1.22cm}{\includegraphics[height=2.454cm]{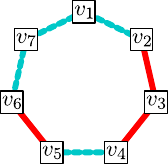}}.\] If we start in the diagram on the left at the replica $\vv_{n-1}={\bf v}_6$ and walk clockwise until reaching ${\bf v}_1$, the replicas other than ${\bf v}_n$ that we must cross are ${\bf v}_4,{\bf v}_5,{\bf v}_3$, so $a_0=4$. If we walk clockwise from ${\bf v}_1$ to ${\bf v}_2$, then we do not pass through any other replicas, so $a_1=1$. If we walk clockwise from ${\bf v}_2$ to ${\bf v}_3$, then the replicas other than ${\bf v}_n$ that we cross are ${\bf v}_6,{\bf v}_4,{\bf v}_5$, so $a_2=4$. Continuing in this fashion, we find that $a_3=4$, $a_4=1$, and $a_5=4$. Thus, our infinite sequence $(a_\beta)_{\beta\geq 0}$ is $4,1,4,4,1,4,4,1,4,\ldots$. This sequence has period $p_\sigma=3$. Also, $\sum_{\ell=0}^{n-2}a_\ell=4+1+4+4+1+4=3(n-1)$, so $m_\sigma=3$. 

We have $\Eflect=\{\{v_2,v_3\},\{v_3,v_4\},\{v_5,v_6\}\}$ and $\Efract=\{\{v_1,v_2\},\{v_4,v_5\},\{v_6,v_7\},\{v_7,v_1\}\}$, so
\[Y_1=\{v_2,v_3,v_4,v_7\}\quad\text{and}\quad Y_{-1}=\{v_1,v_5,v_6\}.\] Thus, $\mu_\sigma=|Y_1|=4$. 
\end{example} 

Note that if the cycle graph $G$ and the partition $\Eflect\sqcup\Efract$ are fixed, then there are only two possible values of $\mu_\sigma$. More precisely, for all $\sigma,\sigma'\in\Lambda_G$, we have either $\mu_\sigma=\mu_{\sigma'}$ or $\mu_\sigma=n-\mu_{\sigma'}$. 

\begin{theorem}\label{thm:cycle}
Let $G=(V,E)$ be a cycle graph with $n$ vertices, and fix a partition $E=\Eflect\sqcup\Efract$ of $E$ into a set of reflection edges and a set of refraction edges. Assume $|\Efract|$ is even. Let $\sigma\in\Lambda_G$. The orbit of $\Theta$ containing $(\sigma,1,1)$ has size \[\frac{np_\sigma}{\gcd(n,\mu_\sigma)}(\mu_\sigma m_\sigma+(n-\mu_\sigma)(n-1-m_\sigma)).\] 
\end{theorem} 

\begin{example}
Suppose $G$, $\Eflect$, $\Efract$, and $\sigma$ are as in \cref{exam:cycle1}. Then $p_\sigma=3$, $m_\sigma=3$, and $\mu_\sigma=4$, so \cref{thm:cycle} tells us that the orbit of $\Theta$ containing $(\sigma,1,1)$ has size $441$. 
\end{example}

Suppose $G$ is an $n$-vertex cycle graph whose edges are all reflection edges. In this case, we have $\mu_\sigma=n$, so the formula in \cref{thm:cycle} simplifies to \[p_\sigma m_\sigma n.\] This determines the sizes of the orbits of (ordinary) toric promotion for a cycle graph, which has not been done previously.   

Another special case worth mentioning is that in which $n$ is even and $G$ is an $n$-vertex cycle graph whose edges are all refraction edges. In this case, we have $\mu_\sigma=n/2$, so the formula in \cref{thm:cycle} simplifies to \[p_\sigma n(n-1).\] Because this formula is independent of $m_\sigma$ and $\mu_\sigma$, we can give a more compact description of the orbit structure of $\Theta$ in \cref{cor:sieving} via the \emph{cyclic sieving phenomenon}.  

Let $X$ be a finite set, and let $\mathcal C_\omega=\langle g\rangle$ be a cyclic group of order $\omega$ that acts on $X$. Let $F(q)\in\mathbb C[q]$. Following Reiner, Stanton, and White \cite{CSP}, we say the triple $(X,\mathcal C_{\omega},F(q))$ \dfn{exhibits the cyclic sieving phenomenon} if for every $k\in\Z$, we have \[F(e^{2\pi ik/\omega})=|\{x\in X:g^k\cdot x=x\}|.\] 

Let $[k]_q=\frac{1-q^k}{1-q}$ and $[k]_q!=[k]_q[k-1]_q\cdots[1]_q$. Given a partition $\lambda=(\lambda_1,\ldots,\lambda_\ell)$ of a positive integer $N$, let $b(\lambda)=\sum_{i=1}^\ell (i-1)\lambda_i$. We also view $\lambda$ as a Young diagram. A \dfn{standard Young tableaux} of shape $\lambda$ is a filling of the boxes of $\lambda$ with the numbers $1,\ldots,N$ so that rows and columns are increasing; let $\SYT(\lambda)$ denote the set of standard Young tableaux of shape $\lambda$. A \dfn{descent} of a tableaux $T\in\mathrm{SYT}(\lambda)$ is an entry $i\in[N-1]$ such that $i+1$ appears in a strictly lower row than $i$ in $T$. The sum of the descents of $T$ is called the \dfn{major index} of $T$ and is denoted $\maj(T)$. Let \[f^\lambda(q)=\sum_{T\in\SYT(\lambda)}q^{\maj(T)}.\] It is known that \[f^\lambda(q)=q^{b(\lambda)}\frac{[N]_q!}{\prod_{\square\in\lambda}[h_\lambda(\square)]_q},\]
where $h_\lambda(\square)$ is the \emph{hook length} of a box $\square$ in $\lambda$ (see \cite[Corollary~7.21.5]{EC2}). Let $f^\lambda_{N\mid\maj}$ denote the number of tableaux $T\in\SYT(\lambda)$ such that $\maj(T)$ is divisible by $N$. Equivalently, \begin{equation}\label{eq:plug_in_roots}
f^\lambda_{N\mid\maj}=\frac{1}{N}\sum_{j=0}^{N-1}f^\lambda(e^{2\pi ij/N}).
\end{equation}

\begin{corollary}\label{cor:sieving}
Suppose $n\geq 4$ is even, and let $G$ be an $n$-vertex cycle graph whose edges are all refraction edges. The sizes of the orbits of $\Theta$ are all divisible by $n(n-1)$. The order of $\Theta^{n(n-1)}$ is $1$ if $n=4$ and is $n-1$ if $n>4$. Let the cyclic group $\mathcal C_{n-1}=\langle g\rangle$ act on $\Xi_G$ by $g\cdot (\sigma,i,\epsilon)=\Theta^{n(n-1)}(\sigma,i,\epsilon)$. Then the triple  
\[\left(\Xi_G,\,\mathcal C_{n-1},\,2n^2(n-1)\sum_{\lambda\vdash n-1}f^\lambda_{n-1\mid\maj}f^\lambda(q)\right)\] exhibits the cyclic sieving phenomenon. 
\end{corollary}

\begin{remark}
Using the Robinson--Schensted correspondence (in a manner similar to \cite{BRS}), one can show that the polynomial in \cref{cor:sieving} is equal to \[2n^2(n-1)\sum_{\substack{w\in\mathfrak S_{n-1} \\ n-1\mid \maj(w^{-1})}}q^{\maj(w)},\] where $\maj(u)$ denotes the major index (i.e., the sum of the descents) of a permutation $u$.   
\end{remark} 

\subsection{Outline} 
In \cref{sec:affine_S_n}, we discuss necessary background concerning the affine symmetric group $\affS_n$. In \cref{sec:billiards}, we introduce refraction billiards and explain how toric promotion with reflections and refractions (and, in particular, ordinary toric promotion) can be seen as a combinatorial billiards system in a torus. \cref{sec:coins} provides a perspective that will allow us to compute orbit sizes of $\Theta$. We prove \cref{thm:forest} in \cref{sec:forests}, and then we prove \cref{thm:cycle,cor:sieving} in \cref{sec:cycles}. 
Finally, we discuss future research directions in \cref{sec:conclusion}.   

\section{The Affine Symmetric Group}\label{sec:affine_S_n} 

Fix an integer $n\geq 3$. Let $\mathfrak S_n$ denote the $n$-th symmetric group, which is the group of permutations of the set $[n]:=\{1,\ldots,n\}$. We sometimes represent permutations in $\mathfrak S_n$ as words in one-line notation. We also frequently identify $[n]$ with $\Z/n\Z$ in the obvious manner and view permutations in $\mathfrak S_n$ as bijections from $[n]$ to $\Z/n\Z$. 

The \dfn{affine symmetric group} $\affS_n$ is the group of bijections $u\colon\Z\to\Z$ such that 
\begin{itemize}
\item $u(i+n)=u(i)+n$ for all $i\in\Z$; 
\item $\sum_{i=1}^nu(i)=\binom{n+1}{2}$. 
\end{itemize} 
An element $u\in\affS_n$ is uniquely determined by where it sends $1,2,\ldots,n$, so we can represent it via its \dfn{window notation} $[u(1),u(2),\ldots,u(n)]$. 
For $a,b\in\Z$ with $a\not\equiv b\pmod{n}$, let us write $r_{a,b}=\prod_{j\in\Z}(a+jn\,\,\,\, b+jn)$, where $(a+jn\,\,\,\, b+jn)$ is the transposition that swaps $a+jn$ and $b+jn$. We have $r_{a,b}\in\affS_n$. Let $\s_i=r_{i,i+1}$. Then $\affS_n$ is a Coxeter group whose \dfn{simple reflections} are $\s_1,\s_2,\ldots,\s_n$ and whose Coxeter graph is a simply laced cycle.  

Consider the $(n-1)$-dimensional Euclidean space \[\UU=\{(x_1,\ldots,x_n)\in\mathbb R^n:x_1+\cdots+x_n=0\}.\] For $k\in\Z$ and distinct $i,j\in[n]$, define the hyperplane \[\HH_{i,j}^k=\{(x_1,\ldots,x_n)\in \UU:x_i-x_j=k\}\subseteq\UU.\] Let $\mathcal H_n=\{\HH_{i,j}^k:1\leq i<j\leq n,\,\, k\in\Z\}$, and let $\HHH=\bigcup_{\HH\in\mathcal H_n}\HH$. There is a natural faithful right action of $\affS_n$ on $\UU$. For $1\leq i\leq n-1$, the simple reflection $\s_i$ acts as the reflection through $\HH_{i,i+1}^0$; the simple reflection $\s_n$ acts as the reflection through $\HH_{1,n}^1$. For each $\HH\in\mathcal H_n$, there is a unique element of $\affS_n$ that acts on $\UU$ as the reflection through $\HH$. 

The closures of the connected components of $\UU\setminus\HHH$ are congruent simplices called \dfn{alcoves}. The \dfn{fundamental alcove} is \[\BB=\{(x_1,\ldots,x_n)\in\UU:x_1\geq x_2\geq\cdots\geq x_n\geq x_1-1\}.\] The (right) action of $\affS_n$ on $\UU$ induces a free and transitive (right) action of $\affS_n$ on the set of alcoves. This allows us to identify elements of $\affS_n$ with alcoves via the map $u\mapsto\BB u$. Two distinct alcoves are \dfn{adjacent} if they have a common facet. For $u\in\affS_n$, the alcoves adjacent to $\BB u$ are precisely $\BB \s_1u,\BB\s_2u,\ldots,\BB \s_{n}u$.

The \dfn{coroot lattice} is the set $Q^{\vee}=\UU\cap\Z^n$ of integer points in $\UU$. For each $\gamma\in Q^{\vee}$, there is an element $t_\gamma\in\affS_n$ that acts on $\UU$ via translation by $\gamma$; that is, $xt_\gamma=x+\gamma$ for all $x\in\UU$. We may identify $Q^{\vee}$ with the normal subgroup $\{t_\gamma:\gamma\in Q^{\vee}\}$ of $\affS_n$. The space $\TT_{n-1}=\UU/Q^{\vee}$ is an $(n-1)$-dimensional torus; let $\quot\colon\UU\to\TT_{n-1}$ denote the natural quotient map. For each $\HH\in\mathcal H_n$, we obtain a toric hyperplane $\quot(\HH)\subseteq \TT_{n-1}$; let $\quot(\mathcal H_n)=\{\quot(\HH):\HH\in\mathcal H_n\}$ be the resulting \emph{toric hyperplane arrangement} (see \cite{Ehrenborg} for more on toric hyperplane arrangements). The closures of the connected components of $\TT_{n-1}\setminus\quot(\HHH)$ are called \dfn{toric regions}. Toric regions are in bijection with orbits of alcoves under the action of $Q^{\vee}$. In \cref{fig:start}, two alcoves in the upper-left image are in the same $Q^{\vee}$-orbit if and only if they have the same color, and the six different colors of alcoves correspond to the six different toric regions shown in the upper-right image. We can view $\mathfrak S_n$ as the subgroup of $\affS_n$ generated by $\s_1,\ldots,\s_{n-1}$. The set $\BB\mathfrak S_n$ is a fundamental domain for the action of $Q^{\vee}$ on $\UU$. There is a bijection $\psi$ from $\mathfrak S_n$ to the set of toric regions defined by $\psi(u)=\quot(\BB u)$. In the image in the upper right of \cref{fig:start}, each permutation $u\in\mathfrak S_3$ is represented as a labeling of a path graph and is drawn inside the toric region $\psi(u)$.  

The affine symmetric group decomposes as the semidirect product $\affS_n\cong\mathfrak S_n\ltimes Q^{\vee}$. It follows that $\affS_n/Q^{\vee}\cong\mathfrak S_n$, and there is a natural quotient map $\affS_n\to\mathfrak S_n$, which we denote by $u\mapsto\overline u$. The one-line notation of the permutation $\overline u$ is obtained by reducing the window notation of $u$ modulo~$n$. For example, if $u=[5,7,1,4,-2]\in\affS_5$, then $\overline u=52143\in\mathfrak S_5$. This group-theoretic quotient map is related to the topological quotient map $\quot$ via the identity $\psi(\overline u)=\quot(\BB u)$. For $i\in\Z/n\Z$, we write $s_i$ for the transposition in $\mathfrak S_n$ that swaps $i$ and $i+1$. Thus, $s_i=\overline{\s_i}$.

\section{Billiards with Reflections and Refractions}\label{sec:billiards} 

In \cref{sec:conclusion}, we will briefly discuss a broad definition of combinatorial billiards with reflections and refractions for arbitrary (finite-rank) Coxeter groups. However, for most of this article, we will focus only on the affine symmetric groups, in which the geometric perspective is especially natural. This approach also has the advantage of being more concrete.  

Let $\mathrm{L}$ be a line in the Euclidean space $\mathbb R^2$. Consider a beam of light that travels in a straight line through $\mathbb R^2$, hits $\mathrm{L}$, and then leaves in a straight line. We allow for three possible scenarios. In one scenario, the beam of light passes directly through $\mathrm{L}$. In the second scenario, the beam of light reflects off of $\mathrm{L}$ so that the angle of incidence equals the angle of reflection. In the third scenario, the beam of light passes through $\mathrm{L}$, but it \dfn{refracts} so that its new direction is opposite to what it would have been if the light beam had reflected.\footnote{When a refraction occurs in our setting, the \emph{refraction coefficient} (see \cite{Baird,Davis1,Davis2}) is $-1$. In the physical world, refraction coefficients are usually positive. However, recent breakthroughs in physics and material science have found certain \emph{metamaterials} that produce negative refraction coefficients \cite{meta}.} See \cref{fig:pass_reflect_refract}.

\begin{figure}[ht]
  \begin{center}
  \includegraphics[height=2cm]{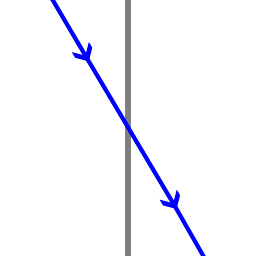}
  \qquad\qquad 
  \includegraphics[height=2cm]{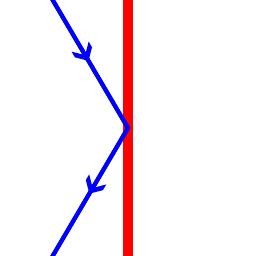} \qquad\quad 
  \includegraphics[height=2cm]{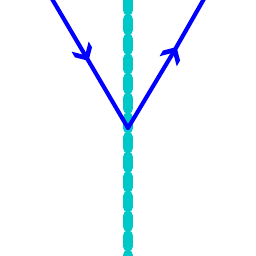}
  \end{center}
\caption{In each image, a beam of light hits a vertical line. On the left, the light beam passes through; in the middle, it reflects; on the right, it refracts.  }\label{fig:pass_reflect_refract}
\end{figure}

We can easily extend our definitions of reflections and refractions to higher dimensions. Let $\mathbb E$ be a $d$-dimensional Euclidean space. Let $\HH$ be an affine hyperplane in $\mathbb E$. Consider a beam of light that travels in a straight line through $\mathbb E$, hits $\HH$ at a point $z$, and then leaves in a straight line. Let $\mathrm{L}^\perp$ be the line passing through $z$ that is orthogonal to $\HH$. We will assume that there is a $2$-dimensional plane $\PP$ containing $\mathrm{L}^\perp$ and the path that the beam of light follows before and after hitting $\HH$. We will also assume that the beam of light only intersects $\HH$ at $z$. Consider the line $\mathrm{L}=\PP\cap \HH$. By restricting our attention to the plane $\PP$, we can speak about the beam of light passing through $\mathrm{L}$, reflecting off of $\mathrm{L}$, or refracting through $\mathrm{L}$, as before. In these three scenarios, we say the beam of light \dfn{passes directly through} $\HH$, \dfn{reflects} off of $\HH$, or \dfn{refracts} through $\HH$, respectively.  

We now describe how to construct a sequence, which we call a \emph{discrete billiards trajectory}, that discretizes a beam of light.

Let $\widetilde\Xi=\affS_n\times\Z/n\Z\times\{\pm 1\}$. For $(u,i,\epsilon)\in\widetilde\Xi$, define $\nu(u,i,\epsilon)$ to be the vector in $Q^{\vee}$ that has entry $\epsilon(1-n)$ in position $\overline u^{-1}(i+\frac{1}{2}(1-\epsilon))$ and has all other entries equal to $\epsilon$. For example, suppose $n=5$, $u=[5,7,1,4,-2]$, $i=2$, and $\epsilon=-1$. Then $\overline u=52143$, so $\overline u^{-1}(i+\frac{1}{2}(1-\epsilon))=\overline u^{-1}(3)=5$. Thus, $\nu(u,i,\epsilon)=(-1,-1,-1,-1,4)$. 

Now define a sequence $u_0,u_1,u_2,\ldots$ with $u_0=u$ and $u_j=s_{i+\epsilon(j-1)}u_{j-1}$ for all $j\geq 1$. Let $\bb$ be a beam of light that starts in the interior of $\BB u$ and travels in the direction of $\nu(u,i,\epsilon)$. One can show that $\bb$ passes through the interiors of the alcoves $\BB u_0,\BB u_1,\BB u_2,\ldots$ (and no others), in this order. See \cref{fig:pass_combinatorial} for an example in $\affS_3$ with $i=1$ and $\epsilon=1$. 

\begin{figure}[ht]
  \begin{center}
  \includegraphics[width=0.47\linewidth]{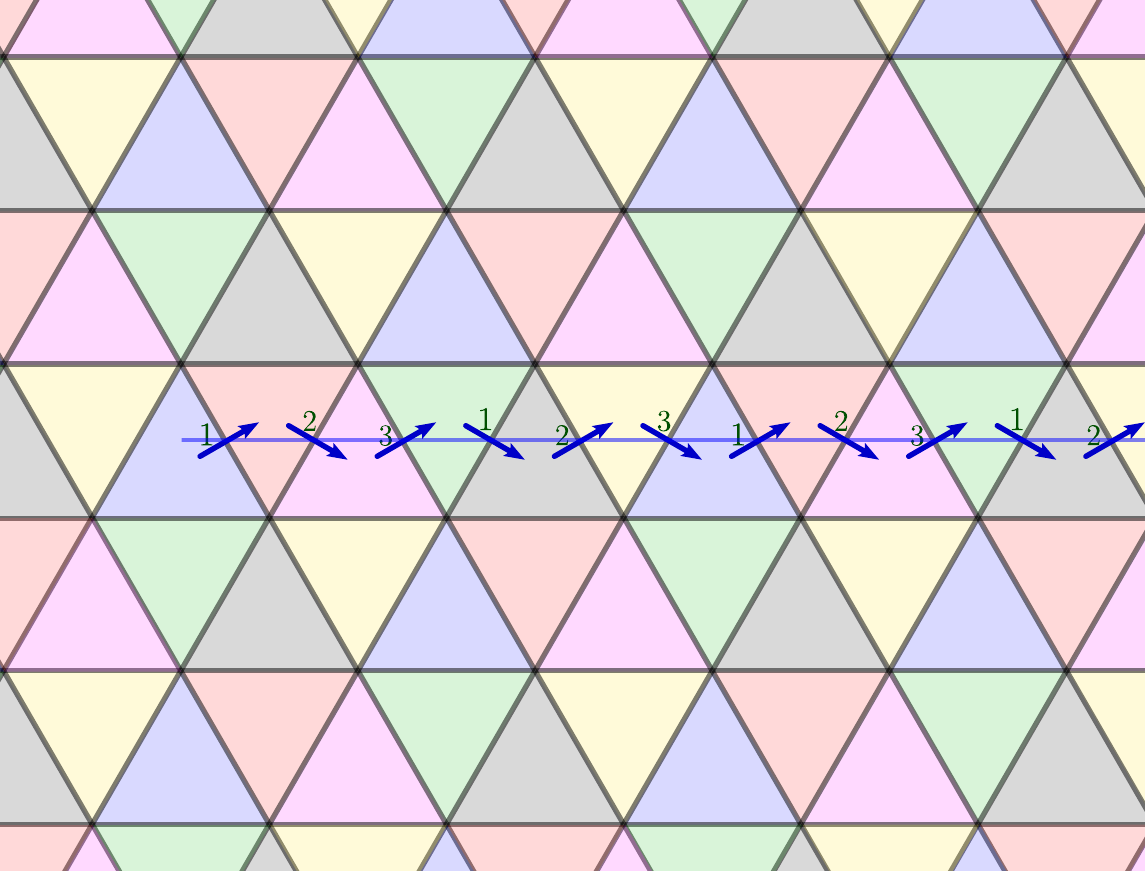} 
  \end{center}
\caption{A sequence of alcoves discretizes a beam of light. Each arrow pointing from a region $w$ to a region $\s_jw$ is labeled with the index $j$.}\label{fig:pass_combinatorial}
\end{figure} 

Let $k$ be a positive integer, and let $\HH^{(u_{k-1},i+\epsilon(k-1))}$ be the hyperplane separating the alcoves $\BB u_{k-1}$ and $\BB u_k$. We can modify the sequence $u_0,u_1,u_2,\ldots$ constructed above to form a new sequence $u_0^{\Reflect},u_1^{\Reflect},u_2^{\Reflect},\ldots$ as follows. Let $u_0^{\Reflect}=u$. Let $u_j^{\Reflect}=s_{i+\epsilon(j-1)}u_{j-1}^{\Reflect}$ for all positive integers $j\neq k$, and let $u_k=u_{k-1}$. Let $\bb^{\Reflect}$ be the beam of light that traces out the same trajectory as $\bb$ until reflecting off of the hyperplane $\HH^{(u_{k-1},i+\epsilon(k-1))}$. Then $\bb^{\Reflect}$ passes through the interiors of the alcoves $\BB u_0^{\Reflect},\BB u_1^{\Reflect},\BB u_2^{\Reflect},\ldots$ (and no others), in this order. See the left side of \cref{fig:reflect_refract_combinatorial}, where $k=7$, $i=1$, $\epsilon=1$, and $\HH^{(u_{6},1)}$ is drawn in thick {\color{red}red}. 

We can also modify the sequence $u_0,u_1,u_2,\ldots$ to form another sequence $u_0^{\Refract},u_1^{\Refract},u_2^{\Refract},\ldots$ as follows. Let $u_j^{\Refract}=u_j$ for all $0\leq j\leq k$. For $j\geq k+1$, let $u_j^{\Refract}=s_{i+\epsilon(2k-1-j)}u_{j-1}^{\Refract}$. Let $\bb^{\Refract}$ be the beam of light that traces out the same trajectory as $\bb$ until refracting through the hyperplane $\HH^{(u_{k-1},i+\epsilon(k-1))}$. Then $\bb^{\Refract}$ passes through the interiors of the alcoves $\BB u_0^{\Refract},\BB u_1^{\Refract},\BB u_2^{\Refract},\ldots$ (and no others), in this order. See the right side of \cref{fig:reflect_refract_combinatorial}, where $k=7$, $i=1$, $\epsilon=1$, and $\HH^{(u_{6},1)}$ is drawn in thick {\color{Teal}teal}. 

\begin{figure}[ht]
  \begin{center}
  \includegraphics[width=0.47\linewidth]{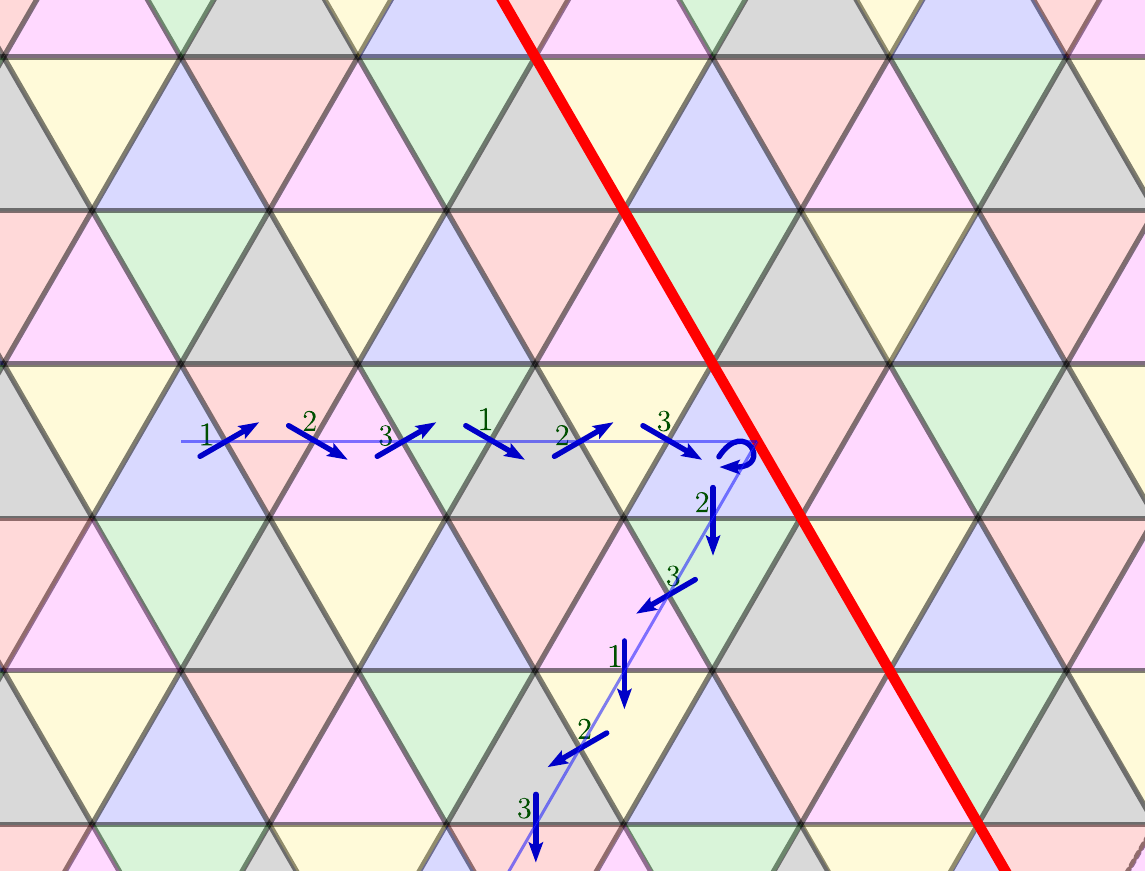} \hspace{0.04\linewidth}   \includegraphics[width=0.47\linewidth]{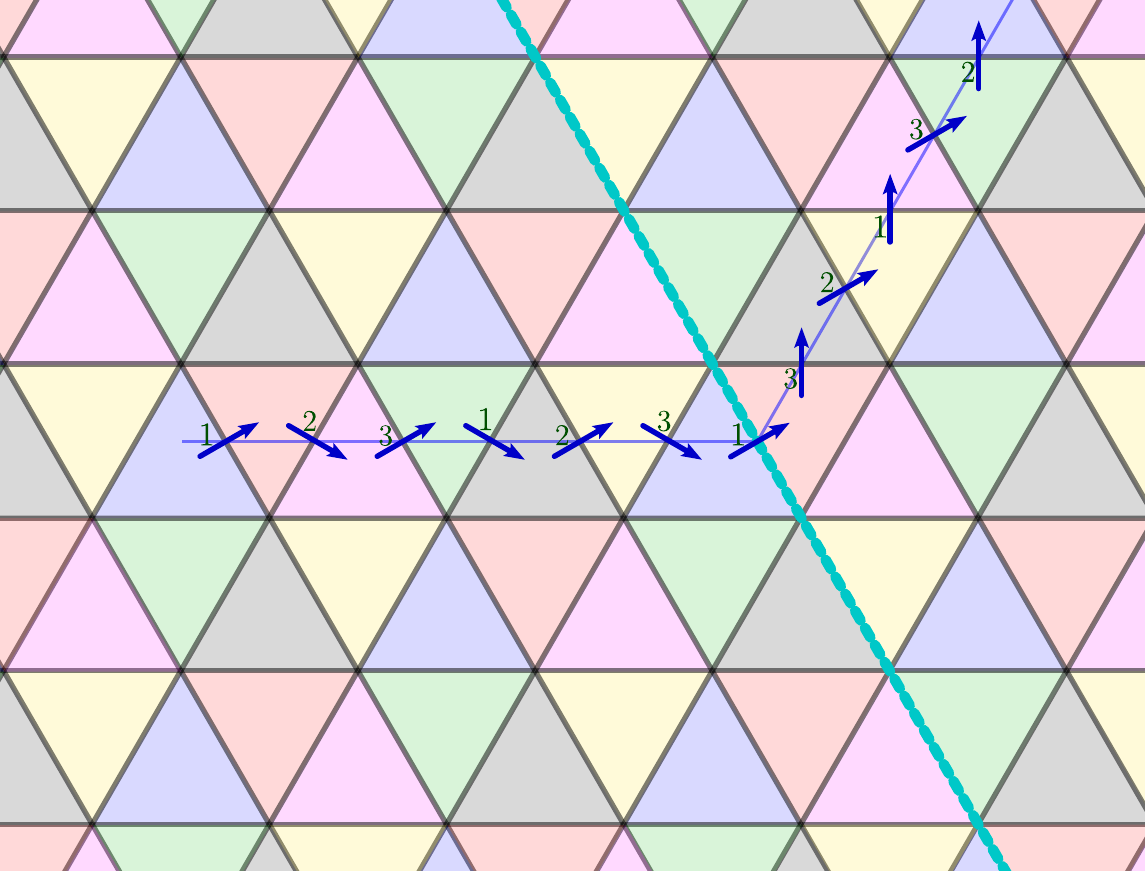} 
  \end{center}
\caption{In both images, a sequence of alcoves discretizes a beam of light, and each arrow pointing from a region $w$ to a region $\s_jw$ is labeled with the index $j$. The loop in the image on the left is not labeled because it corresponds to multiplying by the identity element. On the left, the beam of light reflects; on the right, the beam of light refracts.}\label{fig:reflect_refract_combinatorial}
\end{figure}

Let $\Wflect$ and $\Wfract$ be disjoint subsets of $\mathcal H_n$, and let $\WW=\Wflect\sqcup\Wfract$. Hyperplanes in the sets $\mathcal H_n\setminus\WW$, $\Wflect$, and $\Wfract$ are called \dfn{windows}, \dfn{mirrors}, and \dfn{metalenses}, respectively. For $u\in\affS_n$ and $i\in\Z/n\Z$, let us write $\HH^{(u,i)}$ for the unique hyperplane in $\mathcal H_n$ that separates the alcoves $\BB u$ and $\BB\s_i u$. Let $\widetilde\Xi=\affS_n\times\Z/n\Z\times\{\pm 1\}$. Define a map $\widetilde\Theta\colon\widetilde\Xi\to\widetilde\Xi$ by letting 
\begin{equation}\label{eq:tildeTheta}
\widetilde\Theta(u,i,\epsilon)=\begin{cases} (\s_iu,i+\epsilon,\epsilon) & \mbox{if }\HH^{(u,i)}\not\in \WW; \\   (u,i+\epsilon,\epsilon) & \mbox{if }\HH^{(u,i)}\in \Wflect; \\   (\s_iu,i-\epsilon,-\epsilon) & \mbox{if }\HH^{(u,i)}\in \Wfract.  \end{cases}
\end{equation} 

Start with a triple $(u_0,i_0,\epsilon_0)\in\widetilde\Xi$, and for each positive integer $k$, let $(u_k,i_k,\epsilon_k)=\Theta^k(u_0,i_0,\epsilon_0)$. We call the sequence $u_0,u_1,u_2,\ldots$ a \dfn{discrete billiards trajectory}. The reason for this terminology comes from the fact that there is a beam of light that follows a piecewise linear path, passing through the interiors of the alcoves $\BB u_0,\BB u_1,\BB u_2,\ldots$ (and no others), in this order. Whenever this beam of light hits a window, it passes directly through; whenever it hits a mirror, it reflects; whenever it hits a metalens, it refracts through.\footnote{The articles \cite{Baird,Davis1,Davis2,Jay,Paris} consider \emph{tiling billiards}, in which a bipartite tiling of the plane is fixed and a beam of light refracts (with refraction coefficient $-1$) whenever it passes from one tile to another. Our setting also allows such an interpretation, though possibly in higher dimensions. Indeed, we obtain a tiling of $\mathrm{U}$ whose tiles are just the regions cut out by the metalenses.} 

We now want to project the system constructed above to the torus $\mathbb T_{n-1}$. In the quotient, we want each toric hyperplane to become a \emph{toric window}, a \emph{toric mirror}, or a \emph{toric metalens}. In order for this to make sense, we must assume that we chose the sets $\Wflect$ and $\Wfract$ so that the quotient map $\quot$ does not identify two hyperplanes made of different ``materials.'' In other words, we assume in what follows that every orbit of $\mathcal H_n$ under the action of $Q^\vee$ is contained entirely in one of the sets $\mathcal H_n\setminus\WW$, $\Wflect$, or $\Wfract$. Every such orbit has a unique representative of the form $\HH_{i,j}^0$ with $1\leq i<j\leq n$. Thus, choosing the sets $\Wflect$ and $\Wfract$ is equivalent to choosing a graph $G=([n],E)$ and a partition $E=\Eflect\sqcup\Efract$. Indeed, we simply set \[\Eflect=\{\{i,j\}:\HH_{i,j}^0\in\Wflect\}\quad\text{and}\quad\Efract=\{\{i,j\}:\HH_{i,j}^0\in\Wfract\}.\] By identifying $\Z/n\Z$ with $[n]$ in the obvious manner, we can identify the set $\Lambda_G$ of labelings of $G$ with the symmetric group $\mathfrak S_n$. Then toric promotion with reflections and refactions (defined with respect to $G$) is related to the map $\widetilde\Theta$ via the quotient map from $\affS_n$ to $\mathfrak S_n$. More precisely, if $(u,i,\epsilon)\in\widetilde\Xi$ and $(u',i',\epsilon')=\widetilde\Theta(u,i,\epsilon)$, then $(\overline{u'},i',\epsilon')=\Theta(\overline u,i,\epsilon)$. This is the sense in which toric promotion with reflections and refractions (and, hence, ordinary toric promotion) is really a ``toric combinatorial billiards system.'' 

\cref{fig:beam} shows a beam of light that reflects and refracts in an arrangement of windows, mirrors, and metalenses determined by the graph $\begin{array}{l}\includegraphics[height=0.353cm]{TorefractPIC7}\end{array}$, along with its projection to the torus $\TT_2$. \cref{fig:start} shows the discretization of this beam of light. 

\begin{figure}[ht]
  \begin{center}
  \includegraphics[height=3.484cm]{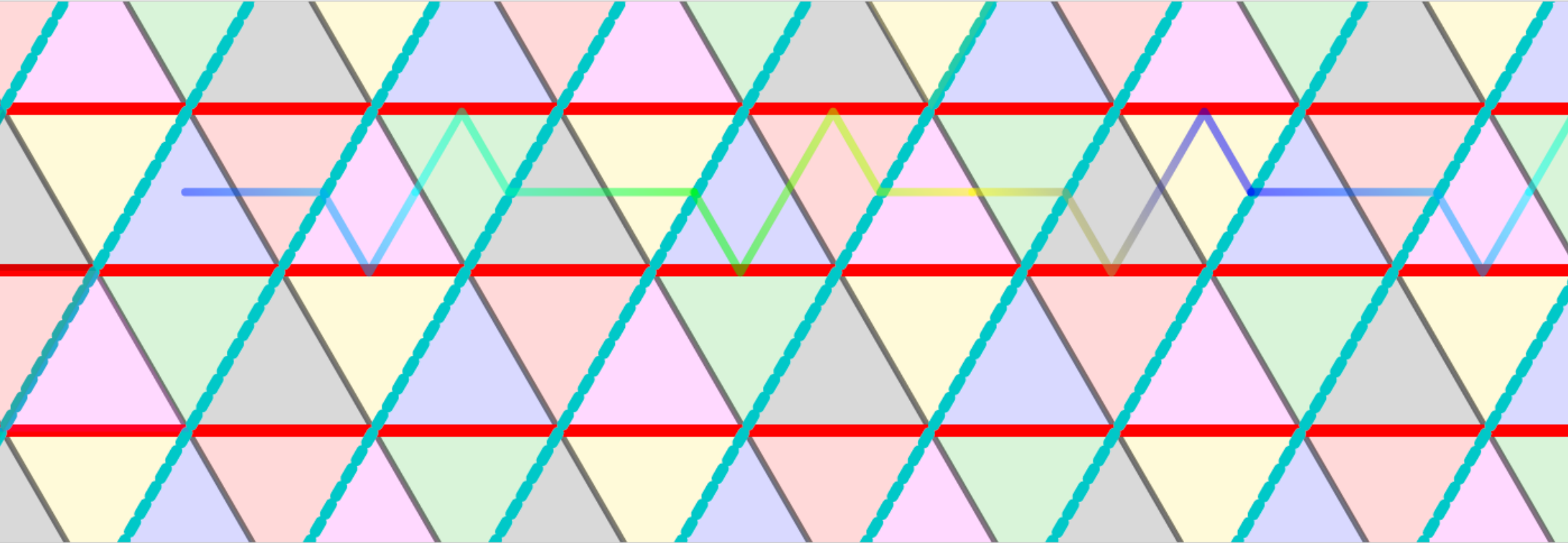}
  \qquad 
  \includegraphics[height=3.484cm]{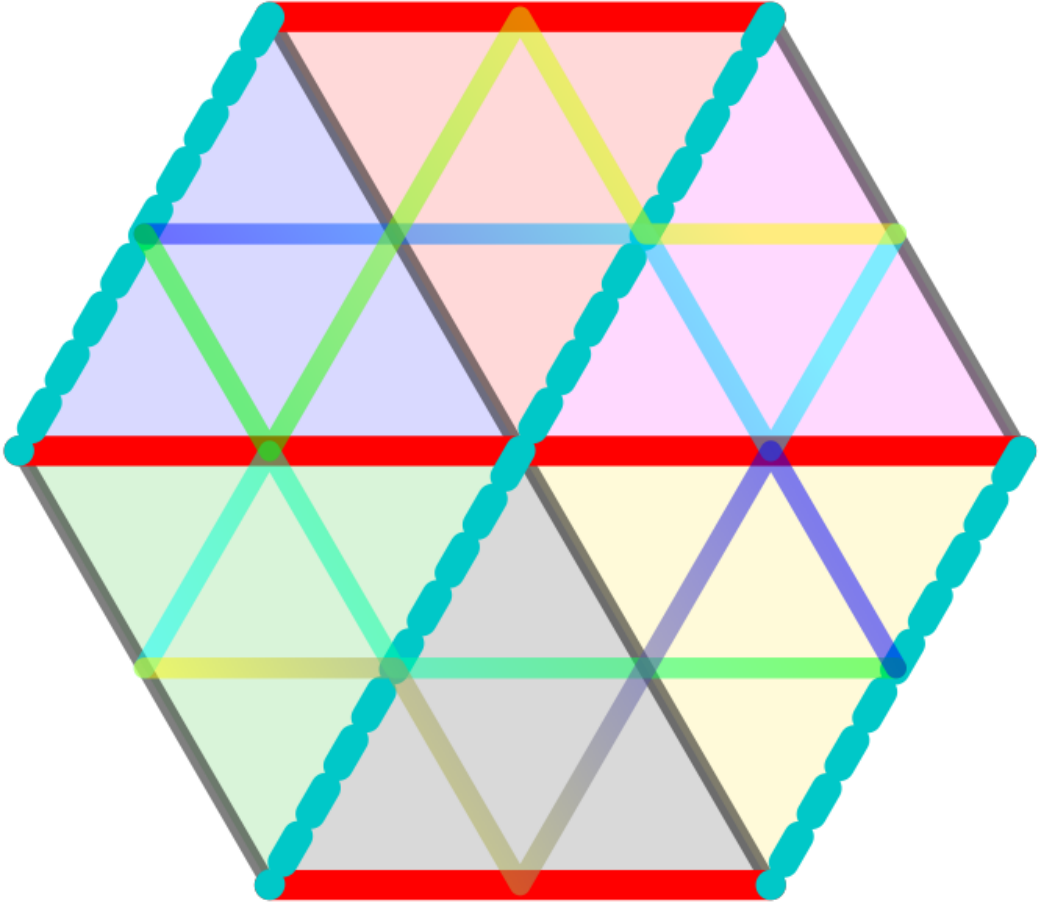} 
  \end{center}
\caption{The image on the left shows a beam of light drawn in a multicolored gradient. The image on the right shows the projection of the beam of light to $\TT_2$. }\label{fig:beam}
\end{figure}

\section{Stones and Coins}\label{sec:coins} 

This brief section presents a use framework, inspired by \cite{DefantPermutoric}, for visualizing the dynamics of toric promotion with reflections and refractions. 

Let $v_1,\ldots,v_n$ be the vertices of $G$. For each vertex $v_k\in V$, we consider a formal symbol $\vv_k$ that we call the \dfn{replica} of $v_k$. Let ${\bf V}=\{\vv_1,\ldots,\vv_n\}$. We represent the labeling $\sigma$ by placing $\vv_k$ on the vertex $\sigma(v_k)$ of $\Cycle_n$ for each $k$. To represent $i$ and $\epsilon$, we also place onto the vertex $i+\frac{1}{2}(1-\epsilon)$ a stone that points clockwise if $\epsilon=1$ and points counterclockwise if $\epsilon=-1$ (so the stone points toward $i+\frac{1}{2}(1+\epsilon)$). This defines the \dfn{stone diagram} of the triple $(\sigma,i,\epsilon)$, which we denote by $\SD(\sigma,i,\epsilon)$. We can also draw the \dfn{coin diagram} of $(\sigma,i,\epsilon)$, which consists of the graph $G$ along with a coin placed on the vertex $\sigma^{-1}(i+\frac{1}{2}(1-\epsilon))$. 

Consider a triple $(\sigma,i,\epsilon)\in\Xi_G$, and let $v_j=\sigma^{-1}(i+\frac{1}{2}(1-\epsilon))$ and $v_k=\sigma^{-1}(i+\frac{1}{2}(1+\epsilon))$. In this situation, we say that $\vv_j$ \dfn{sits on} the stone and that the stone \dfn{points toward} $\vv_k$. When we apply $\Theta$ to $(\sigma,i,\epsilon)$, the stone diagram and coin diagram can change in one of the following ways. 
\begin{itemize}
\item If $\{v_j,v_k\}\not\in E$, then the coin does not move, the replicas $\vv_j$ and $\vv_k$ swap positions, and the stone slides one step in the direction it is pointing. In this case, we imagine that the replica $\vv_j$ rides along with the stone as it slides.
\item If $\{v_j,v_k\}\in\Eflect$, then the coin moves from $v_j$ to $v_k$, and the stone slides from underneath $\vv_j$ to underneath $\vv_k$. 
\item If $\{v_j,v_k\}\in\Efract$, then the coin moves from $v_j$ to $v_k$, the replicas $\vv_j$ and $\vv_k$ swap positions, and the stone reverses its direction.  
\end{itemize}

Define the \dfn{cyclic shift} operator $\cyc\colon\Lambda_G\to\Lambda_G$ by $\cyc(\sigma)(v)=\sigma(v)+1$. We also (slightly abusing notation) define $\cyc\colon\Xi_G\to\Xi_G$ by $\cyc(\sigma,i,\epsilon)=(\cyc(\sigma),i+1,\epsilon)$. We can define $\cyc$ directly on stone diagrams by letting $\cyc(\SD(\sigma,i,\epsilon))=\SD(\cyc(\sigma,i,\epsilon))$. We say a stone diagram $\mathcal D'$ is a \dfn{cyclic rotation} of a stone diagram $\mathcal D$ if $\mathcal D'=\cyc^k(\mathcal D)$ for some integer $k$. 

\section{Forests}\label{sec:forests} 

Assume throughout this section that $G=(V,E)$ is a forest, and fix a partition $E=\Eflect\sqcup\Efract$. Our goal is to prove \cref{thm:forest}, which determines the orbit structure of $\Theta$. 

Let us start with a triple $(\sigma,i,\epsilon)\in\Xi_G$ and repeatedly apply $\Theta$, watching how the associated stone and coin diagrams change over time. Let $(\sigma_t,i_t,\epsilon_t)=\Theta^t(\sigma,i,\epsilon)$. For brevity, let us write $\SD_t=\SD(\sigma_t,i_t,\epsilon_t)$.  

Let $T=(V_T,E_T)$ be the connected component of $G$ containing the vertex $\sigma^{-1}(i+\frac{1}{2}(1-\epsilon))$. At each time step, the coin either does not move or crosses an edge of $G$; therefore, the coin is always on a vertex in $T$. If the coin is on a vertex $v_j$ at time $t$ and is on a different vertex $v_k$ at time $t+1$, then we say the coin moves from $v_j$ to $v_k$ \dfn{at time $t$}. 

As discussed in \cref{sec:intro}, there is a partition $V_T=X_1\sqcup X_{-1}$ such that 
\begin{itemize}
\item every edge in $\Efract\cap E_T$ has one endpoint in $X_1$ and one endpoint in $X_{-1}$; 
\item every edge in $\Eflect\cap E_T$ either has both its endpoints in $X_1$ or has both its endpoints in $X_{-1}$.
\end{itemize}
Let us also assume that $\sigma^{-1}(i_0+\frac{1}{2}(1-\epsilon_0))\in X_{\epsilon_0}$ so that $X_1$ and $X_{-1}$ are uniquely determined. 
Recall that we write $\chi(T;\Eflect,\Efract)=||X_1|-|X_{-1}||$. Because the stone reverses its direction precisely when the coin crosses a refraction edge (and because $T$ is a tree), the direction the stone points is determined by the replica sitting on top of it. Namely, the stone points clockwise whenever the replica of a vertex in $X_1$ sits on it, and it points counterclockwise whenever the replica of a vertex in $X_{-1}$ sits on it. 

Given adjacent vertices $v_j$ and $v_k$ in $T$, let us write $T^{(j,k)}$ for the set of vertices of $T$ that are closer to $v_k$ than to $v_j$. Let $\eta_{j,k}=|T^{(j,k)}|$. 

Consider an edge $\{v_{\ell},v_{\ell'}\}$ in $T$, and let $t$ be a time at which the coin moves from $v_\ell$ to $v_{\ell'}$. Because $\Theta$ is a bijection and $T$ is a tree, there must be some time $t'$ after $t$ at which the coin moves from $v_{\ell'}$ to $v_\ell$. Note that $\SD(s_{i_t}\circ\sigma_t,i_t+\epsilon_t,\epsilon_t)$ is obtained from $\SD_t$ by swapping the positions of $\vv_\ell$ and $\vv_{\ell'}$ and sliding the stone one step in the direction it is pointing. The reader may find it helpful to consult \cref{fig:forest_coins,exam:forest1} while reading the proof of the next lemma. 

\begin{lemma}\label{lem:forest_key_lemma} 
Let $\{v_\ell,v_{\ell'}\}$ be an edge of $T$, and let $t$ be a time at which the coin moves from $v_\ell$ to $v_{\ell'}$. The first time after $t$ at which the coin moves from $v_{\ell'}$ to $v_\ell$ is $t+\eta_{\ell,\ell'}(n-1)$. Moreover, for all $v_j\in T^{(\ell,\ell')}$ and $v_k\in V$ with $j\neq k$, there is a unique time in the interval $[t+1,t+\eta_{\ell,\ell'}(n-1)]$ at which $\vv_j$ sits on the stone and the stone points toward $\vv_k$. In addition, $\SD_{t+\eta_{\ell,\ell'}(n-1)+1}$ is a cyclic rotation of $\SD(s_{i_t}\circ\sigma_t,i_t+\epsilon_t,\epsilon_t)$. \end{lemma}

\begin{proof}
The proof is by induction on $\eta_{\ell,\ell'}$. Assume first that $\eta_{\ell,\ell'}=1$ so that $T^{(\ell,\ell')}=\{v_{\ell'}\}$. Then $v_{\ell'}$ is a leaf of $T$ whose only neighbor is $v_\ell$. At time $t+1$, the stone starts to slide away from $\vv_\ell$, carrying $\vv_{\ell'}$ with it as it slides. The stone slides around $\Cycle_n$ until time $t+n-1$. At time $t+n-1$, the replica $\vv_{\ell'}$ is still sitting on the stone, and the stone is pointing toward $\vv_\ell$. The coin moves from $v_{\ell'}$ to $v_\ell$ at time $t+n-1$. Note that from time $t$ to time $t+n$, the stone either reversed its direction twice (if $\{v_\ell,v_{\ell'}\}\in\Efract$) or zero times (if $\{v_\ell,v_{\ell'}\}\in\Eflect$). Thus, $\epsilon_{t+n}=\epsilon_t$. The rest of the lemma is now straightforward in this case. 

We may now suppose $\eta_{\ell,\ell'}\geq 2$. Note that $v_{\ell'}\in X_{\epsilon_{t+1}}$. We will assume for simplicity that $\epsilon_{t+1}=1$ (so $v_{\ell'}\in X_1$); a virtually identical argument handles the case in which $\epsilon_{t+1}=-1$.
Let $v_{\alpha(1)},\ldots,v_{\alpha(d)}$ be the neighbors of $v_{\ell'}$ in $T$, indexed so that their replicas appear in the clockwise cyclic order $\vv_{\alpha(1)},\ldots,\vv_{\alpha(d)}$ in $\SD_{t+1}$ and so that $\alpha(d)=\ell$. Let us also define $\alpha(0)=\ell$ for simplicity. Let $N[v_{\ell'}]=\{v_{\alpha(1)},\ldots,v_{\alpha(d)},v_{\ell'}\}$ be the closed neighborhood of $v_{\ell'}$ in $T$. For $1\leq j\leq d$, let $\FF_j$ be the set of vertices in $V\setminus N[v_{\ell'}]$ whose replicas appear on the path that goes clockwise from $\vv_{\alpha(j-1)}$ to $\vv_{\alpha(j)}$ in $\SD_{t+1}$. Let ${\FFF}_j$ be the set of replicas of vertices in $\FF_j$. Let $\ff_j=|\FF_j|$. 

Let $t_0=t$. At time $t_0+1$, the stone starts to slide clockwise through the replicas in ${\FFF}_1$. For each $v_k\in \FF_1$, there is a unique time in the interval $[t_0+1,t_0+\ff_1]$ at which $\vv_{\ell'}$ sits on the stone and the stone points toward $\vv_k$. At time $t_0+\ff_1+1$, the coin moves from $v_{\ell'}$ to $v_{\alpha(1)}$. Because $\eta_{\ell',\alpha(1)}<\eta_{\ell,\ell'}$, we can apply our inductive hypothesis to deduce the following. First of all, the first time after $t_0+\ff_1+1$ at which the coin moves from $v_{\alpha(1)}$ to $v_{\ell'}$ is $t_0+\ff_1+\eta_{\ell',\alpha(1)}(n-1)+1$. Second, for all $v_j\in T^{(\ell',\alpha(1))}$ and $v_k\in V$ with $j\neq k$, there is a unique time in the interval ${[t_0+\ff_1+2,t_0+\ff_1+\eta_{\ell',\alpha(1)}(n-1)+1]}$ at which $\vv_j$ sits on the stone and the stone points toward $\vv_k$. Finally, $\SD_{t_0+\ff_1+\eta_{\ell',\alpha(1)}(n-1)+2}$ is a cyclic rotation of 
\begin{equation}\label{eq:proof1}
\SD(s_{t_0+\ff_1+1}\circ\sigma_{t_0+\ff_1+1},i_{t_0+\ff_1+1}+\epsilon_{t_0+\ff_1+1},\epsilon_{t_0+\ff_1+1}).
\end{equation} 
Note that $\epsilon_{t_0+\ff_1+1}=\epsilon_{t_0+1}=\epsilon_{t+1}=1$. Up to cyclic rotation, the stone diagram in \eqref{eq:proof1} is obtained from $\SD_{t_0+1}$ by sliding the stone (along with $\vv_{\ell'}$) clockwise through the replicas in ${\FFF}_1\cup\{\vv_{\alpha(1)}\}$. 

Suppose $d\geq 3$. Let $t_1=t_0+\ff_1+\eta_{\ell',\alpha(1)}(n-1)+1$. At time $t_1+1$, the stone starts to slide clockwise through the replicas in ${\FFF}_2$. For each $v_k\in \FF_2$, there is a unique time in the interval $[t_1+1,t_1+\ff_2]$ at which $\vv_{\ell'}$ sits on the stone and the stone points toward $\vv_k$. At time $t_1+\ff_2+1$, the coin moves from $v_{\ell'}$ to $v_{\alpha(2)}$. Because $\eta_{\ell',\alpha(2)}<\eta_{\ell,\ell'}$, we can again apply induction. We find that the first time after $t_1+\ff_2+1$ at which the coin moves from $v_{\alpha(2)}$ to $v_{\ell'}$ is ${t_1+\ff_2+\eta_{\ell',\alpha(2)}(n-1)+1}$. Also, for all $v_j\in T^{(\ell',\alpha(2))}$ and $v_k\in V$ with $j\neq k$, there is a unique time in the interval ${[t_1+\ff_2+2,t_1+\ff_2+\eta_{\ell',\alpha(2)}(n-1)+1]}$ at which $\vv_j$ sits on the stone and the stone points toward $\vv_k$. Finally, $\SD_{t_1+\ff_2+\eta_{\ell',\alpha(2)}(n-1)+2}$ is a cyclic rotation of  
\begin{equation}\label{eq:proof2}\SD(s_{t_1+\ff_2+1}\circ\sigma_{t_1+\ff_2+1},i_{t_1+\ff_2+1}+\epsilon_{t_1+\ff_2+1},\epsilon_{t_1+\ff_2+1}).
\end{equation} Now, $\epsilon_{t_1+\ff_2+1}=1$. Up to cyclic rotation, the stone diagram in \eqref{eq:proof2} is obtained from the stone diagram at time $t_1+1$ by sliding the stone (along with $\vv_{\ell'}$) clockwise through the replicas in ${\FFF}_2\cup\{\vv_{\alpha(2)}\}$. This, in turn, is a cyclic rotation of the diagram obtained from $\SD_{t_0+1}$ by sliding the stone (along with $\vv_{\ell'}$) clockwise through the replicas in ${\FFF}_1\cup{\FFF}_2\cup\{\vv_{\alpha(1)},\vv_{\alpha(2)}\}$.  

We can continue in this manner, at each step defining 
\begin{equation}\label{eq:proof3}
t_r=t_{r-1}+\ff_r+\eta_{\ell',\alpha(r)}(n-1)+1.
\end{equation} 
Because $\ell=\alpha(d)$, we eventually find that the first time after $t$ at which the coin moves from $v_{\ell'}$ to $v_{\ell}$ is $t_{d-1}+\ff_d+1$. Solving the recurrence in \eqref{eq:proof3} yields that 
\begin{samepage}
\begin{align*}
t_{d-1}+\ff_d+1&=t+(\ff_1+\cdots+\ff_d)+(\eta_{\ell',\alpha(1)}+\cdots+\eta_{\ell',\alpha(d-1)})(n-1)+d \\ 
&=t+|V\setminus N[v_{\ell'}]|+(\eta_{\ell',\alpha(1)}+\cdots+\eta_{\ell',\alpha(d-1)})(n-1)+d \\ 
&= t+n-1+(\eta_{\ell',\alpha(1)}+\cdots+\eta_{\ell',\alpha(d-1)})(n-1) \\ 
&=t+n-1+\left\lvert T^{(\ell',\alpha(1))}\cup\cdots\cup T^{(\ell',\alpha(d-1))}\right\rvert(n-1) \\ 
&= t+n-1+\left\lvert T^{(\ell,\ell')}\setminus\{v_{\ell'}\}\right\rvert(n-1) \\ 
&=t+\eta_{\ell,\ell'}(n-1). 
\end{align*}
\end{samepage}
This proves the first statement of the lemma, and the second statement follows from the above analysis of how the stone diagrams evolve from time $t+1$ to time $t+\eta_{\ell,\ell'}(n-1)$. It also follows from this analysis that, up to cyclic rotation, $\SD_{t+\eta_{\ell,\ell'}(n-1)}$ is obtained from $\SD_{t+1}$ by sliding the stone (along with $\vv_{\ell'}$) clockwise through the replicas in the set \[{\FFF}_1\cup\cdots\cup{\FFF}_{d-1}\cup{\FFF}_d\cup\{\vv_{\alpha(1)},\ldots,\vv_{\alpha(d-1)}\}={\bf V}\setminus\{\vv_{\ell},\vv_{\ell'}\}.\] To prove the last statement of the lemma, we consider two cases, which depend on whether $\{v_\ell,v_{\ell'}\}$ is a reflection edge or a refraction edge. 

Suppose first that $\{v_\ell,v_{\ell'}\}\in\Eflect$. In this case, $\SD_{t+\eta_{\ell,\ell'}(n-1)+1}$ is obtained from $\SD_{t+\eta_{\ell,\ell'}(n-1)}$ by sliding the stone one step clockwise so that it slides from underneath $\vv_{\ell'}$ to underneath $\vv_{\ell}$. Thus, $\SD_{t+\eta_{\ell,\ell'}(n-1)}+1$ is a cyclic rotation of $\SD(s_{i_t}\circ\sigma_t,i_t+\epsilon_t,\epsilon_t)$, as desired. 

Now suppose that $\{v_\ell,v_{\ell'}\}\in\Efract$. In this case, $\SD_{t+\eta_{\ell,\ell'}(n-1)+1}$ is obtained from $\SD_{t+\eta_{\ell,\ell'}(n-1)}$ by swapping the positions of $\vv_{\ell}$ and $\vv_{\ell'}$ and reversing the direction of the arrow. Once again, $\SD_{t+\eta_{\ell,\ell'}(n-1)}+1$ is a cyclic rotation of $\SD(s_{i_t}\circ\sigma_t,i_t+\epsilon_t,\epsilon_t)$. 
\end{proof} 

\begin{figure}[ht]
  \begin{center}
  \includegraphics[width=\linewidth]{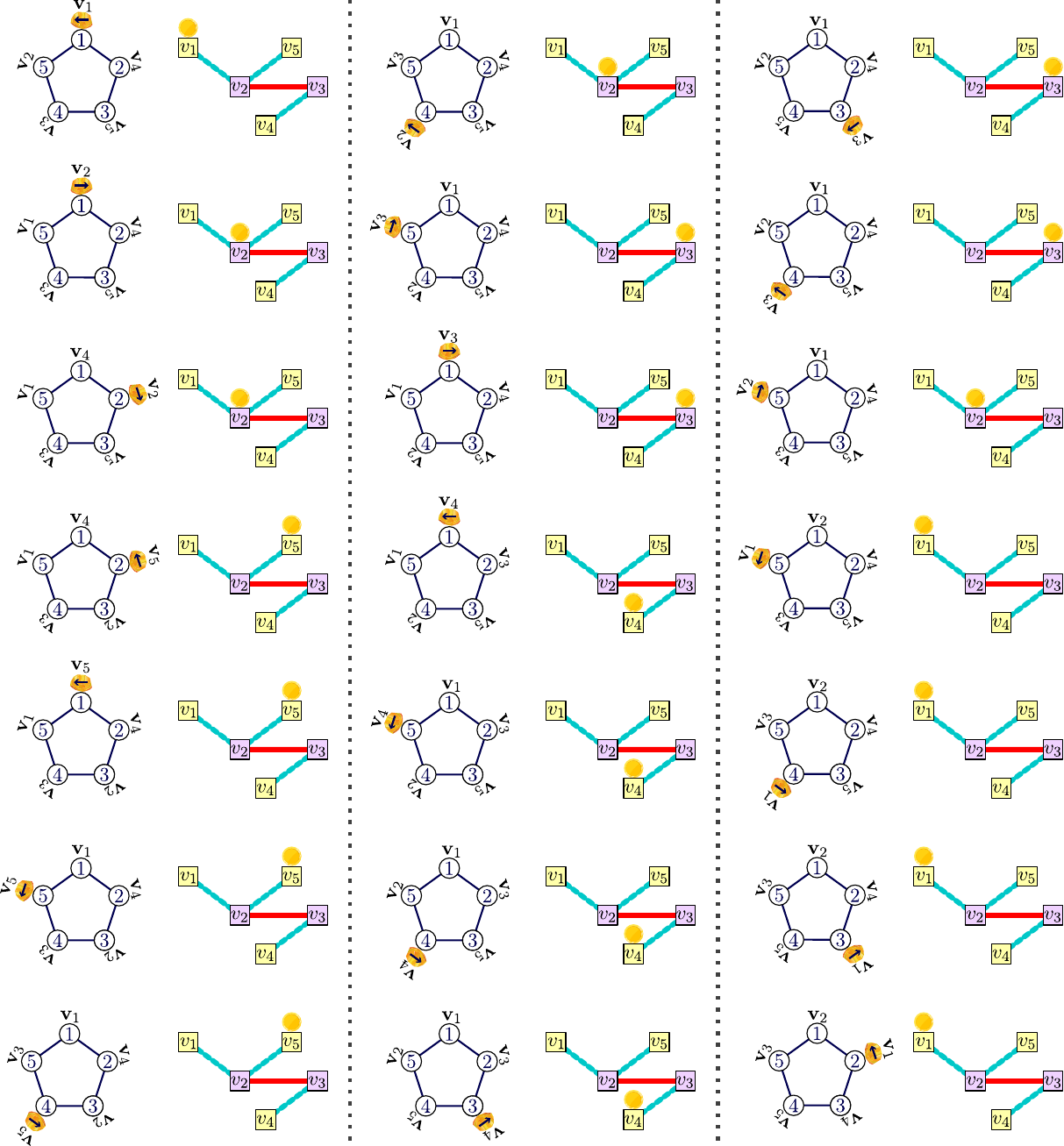}
  \end{center}
\caption{The stone and coin diagrams of triples $(\sigma_t,i_t,\epsilon_t)$ for $0\leq t\leq 20$ (read from top to bottom down each column). Here, $G$ is a tree with $5$ vertices whose only reflection edge is $\{v_2,v_3\}$. In each coin diagram, elements of $X_1$ are indicated in {\color{Purple}purple}, while elements of $X_{-1}$ are indicated in {\color{Yellow}yellow}. }\label{fig:forest_coins}
\end{figure} 

\begin{example}\label{exam:forest1}
\cref{fig:forest_coins} shows stone and coin diagrams that evolve from time $0$ to time $20$. Let $\ell=1$ and $\ell'=2$, and let $t=0$. Then $G=T$, and we have $X_1=\{v_2,v_3\}$ and $X_{-1}=\{v_1,v_4,v_5\}$. In the notation of the proof of \cref{lem:forest_key_lemma}, we have $\eta_{\ell,\ell'}=4$ and $d=3$. Also, \[\alpha(1)=5,\quad\alpha(2)=3,\quad\alpha(0)=\alpha(3)=1,\quad \FF_1=\{v_4\},\quad \FF_2=\varnothing,\quad \FF_3=\varnothing.\] Note that the first time after $t=0$ at which the coin moves from $v_2$ to $v_1$ is $t+\eta_{\ell,\ell'}(n-1)=16$. Moreover, $\SD_{17}$ is equal to (hence, is a cyclic rotation of) $\SD(s_{i_0}\circ\sigma_{0},i_0+\epsilon_0,\epsilon_0)$. 
\end{example}

Now, let $\{v_\ell,v_{\ell'}\}$ be an edge of $T$, and assume that $v_{\ell}$ is a leaf of $T$. Then $\eta_{\ell,\ell'}=|V_T|-1$. It follows from the proof of \cref{lem:forest_key_lemma} that there is a time $t^*$ at which the coin moves from $v_\ell$ to $v_{\ell'}$. \cref{lem:forest_key_lemma} tells us that the first time after $t^*$ at which the coin moves from $v_{\ell'}$ to $v_\ell$ is $t^*+(|V_T|-1)(n-1)$. The lemma also tells us that $\SD_{t^*+(|V_T|-1)(n-1)+1}$ is a cyclic rotation of $\SD(s_{i_{t^*}}\circ\sigma_{t^*},i_{t^*}+\epsilon_{t^*},\epsilon_{t^*})$. At time $t^*+(|V_T|-1)(n-1)+1$, the stone starts to slide away from $\vv_{\ell'}$, carrying $\vv_{\ell}$ with it as it slides. The stone slides around $\Cycle_n$ until time $t^*+|V_T|(n-1)$. At time $t^*+|V_T|(n-1)$, the replica $\vv_{\ell}$ is sitting on the stone, and the stone is pointing toward $\vv_{\ell'}$. It follows that $|V_T|(n-1)$ is the smallest positive integer $r$ such that $\SD_{t^*+r}$ is a cyclic rotation of $\SD_{t^*}$. Let $\delta$ be the smallest nonnegative integer such that $\SD_{t^*+|V_T|(n-1)}=\cyc^\delta(\SD_{t^*})$. Repeating the preceding argument, we find that $\SD_{t^*+K|V_T|(n-1)}=\cyc^{K\delta}(\SD_{t^*})$ for every positive integer $K$. Thus, the orbit of $\Theta$ containing $(\sigma_{t^*},i_{t^*},\epsilon_{t^*})$, which is also the orbit containing $(\sigma_0,i_0,\epsilon_0)$, has size $|V_T|n(n-1)/\gcd(n,\delta)$. To complete the proof of \cref{thm:forest}, we just need to show that 
\begin{equation}\label{eq:delta}
\gcd(n,\delta)=\gcd(n,\chi(T;\Eflect,\Efract))
\end{equation} 

It follows from \cref{lem:forest_key_lemma} and the previous paragraph that for all $v_j\in V_T$ and $v_k\in V$ with $j\neq k$, there is a unique time $\gamma(j,k)\in[t^*,t^*+|V_T|(n-1)-1]$ at which $\vv_j$ sits on the stone and the stone points toward $\vv_k$. The stone moves clockwise at time $\gamma(j,k)$ if and only if $v_{j}\in X_1$ and $\{v_j,v_k\}\not\in\Efract$. Since each refraction edge in $T$ has one endpoint in $X_1$ and one endpoint in $X_{-1}$, the number of times in the interval $[t^*,t^*+|V_T|(n-1)-1]$ when the stone moves clockwise is $|X_1|(n-1)-|\Efract\cap E_T|$. Similarly, the number of times in the interval $[t^*,t^*+|V_T|(n-1)-1]$ when the stone moves counterclockwise is $|X_{-1}|(n-1)-|\Efract\cap E_T|$. Therefore, 
\begin{equation}\label{eq:delta_congruence}
\delta\equiv (|X_{1}|(n-1)-|\Efract\cap E_T|)-(|X_{-1}|(n-1)-|\Efract\cap E_T|)=-(|X_1|-|X_{-1}|)\pmod{n}.
\end{equation} 
This implies that \eqref{eq:delta} holds, which completes proof of \cref{thm:forest}.  

\begin{example}\label{exam:forest2}
Let us set $\ell=1$ and $\ell'=2$ in the example shown in \cref{fig:forest_coins}. Here, we have $G=T$ and $|V_T|=5$. The coin moves from $v_\ell$ to $v_{\ell'}$ at time $t^*=0$. The first time after $0$ at which the stone diagram is a cyclic rotation of $\SD_0$ is $t^*+|V_T|(n-1)=20$. Furthermore, $\SD_{20}$ is obtained by rotating $\SD_0$ clockwise by $\delta=1$ step. Since $X_1=\{v_2,v_3\}$ and $X_{-1}=\{v_1,v_4,v_5\}$, we have $\delta\equiv -(|X_1|-|X_{-1}|)\pmod{n}$, as stated in \eqref{eq:delta_congruence}.  
\end{example}

\section{Cycles}\label{sec:cycles} 

We now assume $G$ is a cycle graph with $n$ vertices and an even number of refraction edges. 

Start with a triple $(\sigma,1,1)\in\Xi_G$. Assume $G$ is embedded in the plane so that $v_1,\ldots,v_n$ is the clockwise cyclic ordering of $V$. As we iteratively apply $\Theta$, the coin always moves in one direction (either clockwise or counterclockwise) around $G$. By renaming the vertices of $G$ and choosing a different embedding into the plane if necessary, we may assume that $\sigma(v_n)=1$, that $\sigma(v_1)<\sigma(v_{n-1})$ when we identify $\Z/n\Z$ with $[n]$ in the obvious manner, and that the coin moves clockwise. 

Consider the stone diagram $\SD(\sigma,1,1)$. Let $a_0$ be $1$ more than the number of replicas other than $\vv_n$ (not including $\vv_{n-1}$ or $\vv_1$) that we cross while walking clockwise from $\vv_{n-1}$ to $\vv_1$. For $1\leq k\leq n-2$, let $a_k$ be $1$ more than the number of replicas other than $\vv_n$ that we cross while walking clockwise from $\vv_{k}$ to $\vv_{k+1}$. Now construct an infinite sequence $(a_\beta)_{\beta\geq 0}$ by declaring that $a_{j+n-1}=a_j$ for all $j\geq 0$. Let $p_\sigma$ be the period of this sequence, and let $m_\sigma$ be the positive integer such that $\sum_{\ell=0}^{n-2}a_\ell=m_\sigma(n-1)$. Note that 
\begin{equation}\label{eq:sum_of_as}
\sum_{\ell=k}^{p_\sigma R+k-1}a_{\ell}=p_\sigma m_\sigma R
\end{equation} for all nonnegative integers $R$ and $k$.
Because $G$ has an even number of refraction edges, there is a unique partition $V=Y_1\sqcup Y_{-1}$ of the vertex set of $G$ such that 
\begin{itemize}
\item every refraction edge has one endpoint in $Y_1$ and one endpoint in~$Y_{-1}$; 
\item every reflection edge either has both its endpoints in $Y_1$ or has both its endpoints in~$Y_{-1}$; 
\item $v_n\in Y_1$.
\end{itemize}
Let $\mu_\sigma=|Y_1|$. The stone points clockwise when the replica of an element of $Y_1$ sits on it, and it points counterclockwise when the replica of an element of $Y_{-1}$ sits on it.  

Our goal is to show that the size of the orbit of $\Theta$ containing $(\sigma,1,1)$ is given by the formula in \cref{thm:cycle}. 

Let $b$ be the smallest nonnegative integer such that $\sigma(v_{n-1})\equiv -b\pmod{n}$. Then $\Theta^{-b}(\sigma,1,1)$ has the form $(\sigma',1-b,1)$, and $\SD(\sigma',1-b,1)$ is obtained from $\SD(\sigma,1,1)$ by sliding the stone along with $\vv_n$ counterclockwise $b$ spaces. If we define the sequence $(a_\beta)_{\beta\geq 0}$ using the stone diagram $\SD(\sigma',1-b,1)$, then we get the same result as when we defined it using $\SD(\sigma,1,1)$. Therefore, replacing $(\sigma,1,1)$ with $(\sigma',1-b,1)$ if necessary, we may assume that $\sigma(v_{n-1})=\sigma(v_n)-1$. Let $(\sigma_t,i_t,\epsilon_t)=\Theta^{t-1}(\sigma,1,1)$. Then the coin moves from $v_{n-1}$ to $v_n$ at time $0$.

Let us consider the indices of the vertices of $G$ and their replicas modulo $n$ by declaring that $v_{j+n}=v_j$ and $\vv_{j+n}=\vv_j$. Let $t_0^*<t_1^*<\cdots$ be the nonnegative times at which the coin moves. Thus, the coin moves from $v_{k-1}$ to $v_k$ at time $t_k^*$. Note that $t_0^*=0$. Let $Z$ be the set of nonnegative integers $j$ such that $v_j\in Y_1$. Let us write $Z=\{z_0<z_1<\cdots\}$. In particular, $z_0=0$. 

Let $\mathcal D_0=\SD_1$. Let $\mathcal D_1$ be the stone diagram obtained from $\mathcal D_0$ by sliding the stone with $\vv_0$ clockwise through $a_0-1$ replicas and then sliding the stone from underneath $\vv_0$ to underneath $\vv_1$. Let us recursively define $\mathcal D_r$ to be the stone diagram obtained from $\mathcal D_{r-1}$ by sliding the stone with $\vv_{r-1}$ clockwise through $a_{r-1}-1$ replicas and then sliding the stone from underneath $\vv_{r-1}$ to underneath $\vv_r$. (See \cref{fig:cycle_example}.) 

\begin{lemma}\label{lem:cycle}
For each nonnegative integer $M$, we have $\mathcal D_{p_\sigma Mn}=\SD_1$.  
\end{lemma}

\begin{proof}
We can view $\mathcal D_{p_\sigma Mn}$ as the stone diagram of a triple $(\sigma',i',\epsilon')\in\Xi_G$. The replica $\vv_0$ sits on the stone and the stone points clockwise in $\mathcal D_{p_\sigma Mn}$. Therefore, we just need to show that $\sigma'=\sigma$. 

Consider how the replica $\vv_0$ moves when we transform $\mathcal D_{r-1}$ into $\mathcal D_r$. If $r-1=\ell n$ for some nonnegative integer $\ell$, then $\vv_0$ rides on the stone $a_{\ell n}-1$ steps clockwise; otherwise, $\vv_0$ does not ride on the stone. Overall, the net amount that $\vv_0$ rides clockwise on the stone when we change $\mathcal D_0$ to $\mathcal D_{p_\sigma Mn}$ is \[\sum_{\ell=0}^{p_\sigma M-1}(a_{\ell n}-1)=\sum_{\ell =0}^{p_\sigma M-1}(a_{\ell}-1)=p_\sigma M(m_\sigma-1),\] where we have used \eqref{eq:sum_of_as} (with $R=M$ and $k=0$) and the fact that $a_{\ell n}=a_\ell$. Note, however, that $\vv_0$ can also move when it is not riding on the stone. Namely, $\vv_0$ will move $1$ step counterclockwise whenever the stone slides through it. This occurs $m_\sigma-1$ times when we transform $\mathcal D_{\ell n}$ into $\mathcal D_{(\ell+1)n}$. Therefore, the total number of times that the stone slides through $\vv_0$ when we transform $\mathcal D_0$ into $\mathcal D_{p_\sigma Mn}$ is $p_\sigma M(m_\sigma-1)$. It follows that the position of $\vv_0$ in $\mathcal D_{p_\sigma Mn}$ is the same as the position of $\vv_0$ in the diagram $\mathcal D_0=\SD_1$. That is, $\sigma'(v_0)=\sigma(v_0)$. A similar argument shows that $\sigma'(v_j)=\sigma(v_j)$ for every vertex $v_j$. 
\end{proof} 

Let us now return to analyzing how the stone and coin diagrams of the triples $(\sigma_t,i_t,\epsilon_t)$ evolve over time. We have $z_0=0$ and $t_{z_0}^*=t_0^*=0$. At time $t_{z_0}^*+1$, the coin sits on $v_{z_0}$, so the replica $\vv_{z_0}$ sits on the stone. Let us watch what happens to the stone diagrams from time $t_{z_0}^*+1$ to time $t_{z_{1}}^*$. First, the stone slides clockwise along with $\vv_{z_0}$ through $a_{z_0}-1$ replicas. At time $t_{z_0}^*+a_{z_0}$, the coin moves from $v_{z_0}$ to $v_{z_0+1}$; thus $t_{z_0+1}^*=t_{z_0}^*+a_{z_0}$. If $z_{0}+1<z_{1}$, then at time $t_{z_0+1}^*$, the replicas $\vv_{z_0}$ and $\vv_{z_0+1}$ swap positions, and the stone reverses direction; the stone then slides counterclockwise with $\vv_{z_0+1}$ through $n-1-a_{z_0+1}$ replicas. If $z_0+2<z_{1}$, then at time $t_{z_0+2}^*=t_{z_0+1}^*+(n-1+a_{z_0+1})$, the stone slides from underneath $\vv_{z_0+1}$ to underneath $\vv_{z_0+2}$; the stone then slides counterclockwise with $\vv_{z_{0}+2}$ through $n-1-a_{z_0+2}$ replicas until the coin moves again at time $t_{z_0+3}^*=t_{z_0+2}^*+(n-1-a_{z_0+2})$. This continues until the stone eventually slides counterclockwise with $\vv_{z_{1}-1}$ until reaching $\vv_{z_{1}}$. At time $t_{z_{1}}^*$, the coin moves from $v_{z_{1}-1}$ to $v_{z_{1}}$. At time $t_{z_{1}}^*+1$, the replica $\vv_{z_{1}}$ sits on the stone, and the stone once again points clockwise. It follows from this analysis that $t_{z_1}^*=t_{z_0}^*+a_{z_0}+\sum_{\beta=z_0+1}^{z_1-1}(n-1-a_{\beta})$. Moreover, it is straightforward to check that $\SD_{t_{z_1}^*+1}=\cyc^{z_1-1}(\mathcal D_{z_1})$. 

The preceding paragraph generalizes if we simply replace $z_0$ and $z_1$ everywhere by $z_r$ and $z_{r+1}$ (for $r$ a nonnegative integer). In summary, we find that \begin{equation}\label{eq:t_r_recurrence}
t_{z_r+1}^*=t_{z_r}^*+a_{z_r}+\sum_{\beta=z_r+1}^{z_{r+1}-1}(n-1-a_{\beta})
\end{equation} and that 
\begin{equation}\label{eq:calD}
\SD_{t_{z_{r+1}}^*+1}=\cyc^{z_{r+1}-r-1}(\mathcal D_{z_{r+1}}).
\end{equation} 

\begin{figure}[ht]
  \begin{center}
  \includegraphics[width=\linewidth]{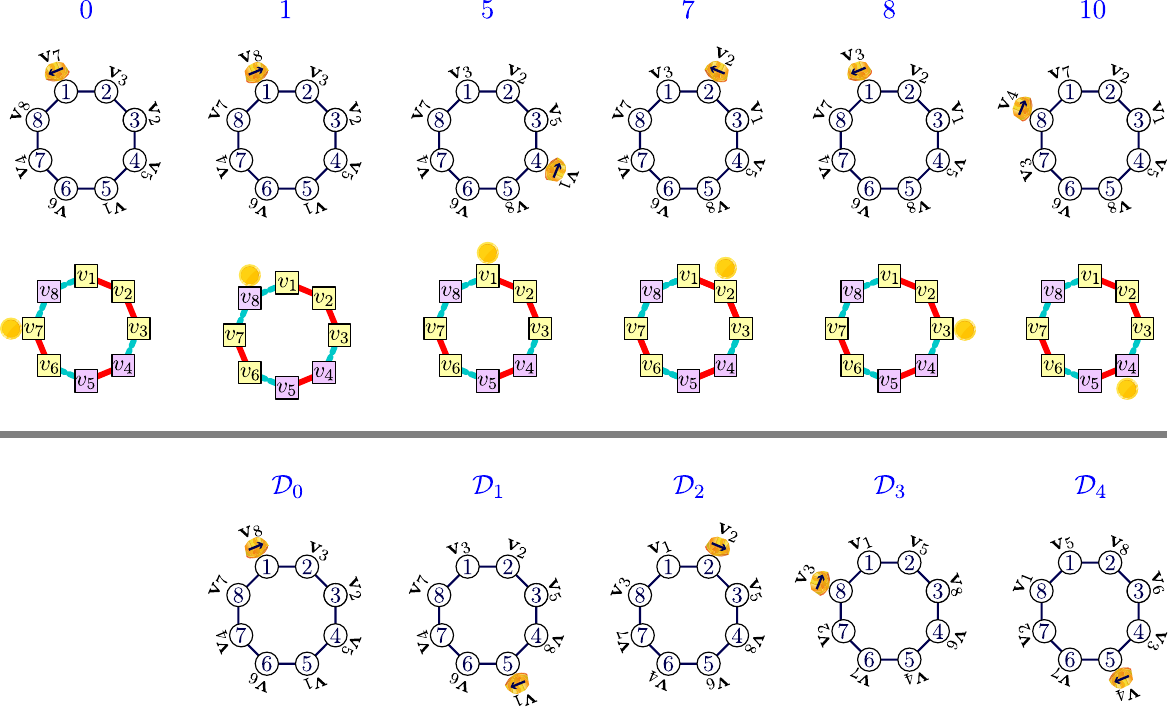}
  \end{center}
\caption{On the top are stone diagrams and coin diagrams at times $0,1,5,7,8,10$. On the bottom are stone diagrams $\mathcal D_0,\mathcal D_1,\mathcal D_2,\mathcal D_3,\mathcal D_4$. In each coin diagram, elements of $Y_1$ are indicated in {\color{Purple}purple}, while elements of $Y_{-1}$ are indicated in {\color{Yellow}yellow}.}\label{fig:cycle_example}
\end{figure} 

\begin{example}
The top part of \cref{fig:cycle_example} shows stone and coin diagrams at times $0,1,5,7,8,10$. The stone diagram at time $1$ is $\SD_1=\SD(\sigma,1,1)$. In $G$, the reflection edges are $\{v_1,v_2\}$, $\{v_2,v_3\}$, $\{v_4,v_5\}$, $\{v_6,v_7\}$, and the refraction edges are $\{v_3,v_4\},\{v_5,v_6\},\{v_7,v_8\},\{v_8,v_1\}$. We have ${Y_1=\{v_4,v_5,v_8\}}$ and $Y_{-1}=\{v_1,v_2,v_3,v_6,v_7\}$, so $\mu_\sigma=|Y_1|=3$. The sequence $(a_\beta)_{\beta\geq 0}$ starts with $4,5,6,5,4,2,2$ and has period $p_\sigma=7$. Since $\sum_{\ell=0}^6a_\ell=28=4(n-1)$, we find that $m_\sigma=4$. We have $z_0=0$ and $z_1=4$. Also, $t_0^*=0$, $t_1^*=t_0^*+a_0=4$, $t_2^*=t_1^*+(n-1-a_1)=6$, $t_3^*=t_2^*+(n-1-a_2)=7$, and $t_4^*=t_3^*+(n-1-a_3)=9$. In particular, \[t_{z_1}^*=t_4^*=0+4+(n-1-5)+(n-1-6)+(n-1-5)=t_{z_0}^*+a_{z_0}+\sum_{\beta=z_0+1}^{z_1-1}(n-1-a_\beta).\] The bottom part of \cref{fig:cycle_example} shows the stone diagrams $\mathcal D_0,\mathcal D_1,\mathcal D_2,\mathcal D_3,\mathcal D_4$. Observe that we have $\SD_{t_{z_1}^*+1}=\SD_{10}=\cyc^{3}(\mathcal D_{z_1})$. 
\end{example}

Let $M$ be a positive integer. Since $v_{p_\sigma Mn}=v_0\in Y_1$, we know that $p_\sigma Mn\in Z$. In fact, ${p_\sigma Mn=z_{p_\sigma M\mu_\sigma}}$. Setting $r=p_\sigma M\mu_\sigma-1$ in \eqref{eq:calD} and invoking \cref{lem:cycle}, we find that 
\begin{equation}\label{eq:MM}
\SD_{t_{p_\sigma Mn}^*+1}=\cyc^{p_\sigma M(n-\mu_\sigma)}(\mathcal D_{p_\sigma Mn})=\cyc^{-p_\sigma M\mu_\sigma}(\mathcal D_{p_\sigma Mn})=\cyc^{-p_\sigma M\mu_\sigma}(\SD_1). 
\end{equation} 

Let $\mathrm{T}$ be the size of the orbit of $\Theta$ containing $(\sigma,1,1)$. Then $\SD_{\mathrm{T}}=\SD_0$, so the coin must move from $v_{n-1}$ to $v_n$ at time $\mathrm{T}$. Thus, $\mathrm{T}=t_{Kn}^*$ for some positive integer $K$. Note that $Kn$ must be divisible by the period $p_\sigma$ of the sequence $(a_\beta)_{\beta\geq 0}$. Since $p_\sigma$ divides $n-1$, it must be coprime to $n$, so $K$ must be divisible by $p_\sigma$. Thus, $K=p_\sigma M$, where $M$ is the smallest positive integer such that $\SD_{t_{p_\sigma Mn}^*+1}=\SD_1$. It follows from \eqref{eq:MM} that $M=n/\gcd(n,p_\sigma\mu_\sigma)$. Because $p_\sigma$ is coprime to $n$, we have $M=n/\gcd(n,\mu_\sigma)$. In summary, we find that $\mathrm{T}=t_L^*$, where $L=p_\sigma n^2/\gcd(n,\mu_\sigma)$. 

Let $[0,n-1]=\{0,1,\ldots,n-1\}$, and note that $|[0,n-1]\cap Z|=\mu_\sigma$. Since $L$ is divisible by $n$, it follows from \eqref{eq:t_r_recurrence} that 
\begin{align*}
t_L^*&=\sum_{\substack{0\leq j\leq L-1 \\ v_{j}\in Y_1}}a_j+\sum_{\substack{0\leq j\leq L-1 \\ v_{j}\in Y_{-1}}}(n-1-a_j) \\ 
&=\sum_{k\in[0,n-1]\cap Z}\sum_{\beta=0}^{L/n-1}a_{\beta n+k}+\sum_{k\in[0,n-1]\setminus Z}\sum_{\beta=0}^{L/n-1}(n-1-a_{\beta n+k}) \\ 
&=\sum_{k\in[0,n-1]\cap Z}\sum_{\beta=0}^{L/n-1}a_{\beta+k}+\sum_{k\in[0,n-1]\setminus Z}\sum_{\beta=0}^{L/n-1}(n-1-a_{\beta+k}) \\ 
&=\sum_{k\in[0,n-1]\cap Z}\sum_{\beta=0}^{L/n-1}a_{\beta+k}-\sum_{k\in[0,n-1]\setminus Z}\sum_{\beta=0}^{L/n-1}a_{\beta+k}+(n-\mu_\sigma)(L/n)(n-1)
\end{align*} 
where we have used the fact that $a_{\beta n+k}=a_{\beta+k}$. Since $L/n$ is divisible by $p_\sigma$, we can employ \eqref{eq:sum_of_as} to find that 
$\sum_{\beta=0}^{L/n-1}a_{\beta+k}=m_\sigma L/n$ for every nonnegative integer $k$. Consequently, 
\begin{align*}
t_L^*&=\mu_\sigma m_\sigma L/n-(n-\mu_\sigma)m_\sigma L/n+(n-\mu_\sigma)(L/n)(n-1) \\ 
&=\frac{np_\sigma}{\gcd(n,\mu_\sigma)}(\mu_\sigma m_\sigma+(n-\mu_\sigma)(n-1-m_\sigma)). 
\end{align*} 
This proves \cref{thm:cycle}. 

Let us now assume that $n$ is even and that all edges of $G$ are refraction edges; we aim to prove \cref{cor:sieving}. For $i\in\Z/n\Z$ and $\epsilon\in\{\pm 1\}$, define $\boldsymbol{\omega}_{i,\epsilon}\colon\Z/n\Z\to\Z/n\Z$ by ${\boldsymbol{\omega}_{i,\epsilon}(j)=\epsilon(j-i)+1}$. Define a map $\varphi\colon\Xi_G\to\Lambda_G\times\{1\}\times\{1\}$ by $\varphi(\rho,i,\epsilon)=(\boldsymbol{\omega}_{i,\epsilon}\circ\rho,1,1)$. The orbit of $\Theta$ containing $(\rho,i,\epsilon)$ has the same size as the orbit of $\Theta$ containing $\varphi(\rho,i,\epsilon)$. For every integer $\kappa$, since each element of $\Lambda_G\times\{1\}\times\{1\}$ has $2n$ preimages under $\varphi$, we have  
\begin{equation}\label{eq:bold_omega}
|\{(\rho,i,\epsilon)\in\Xi_G:\Theta^\kappa(\rho,i,\epsilon)=(\rho,i,\epsilon)\}|=2n|\{\sigma\in\Lambda_G:\Theta^\kappa(\sigma,1,1)=(\sigma,1,1)\}|. 
\end{equation}
 
For any $\sigma\in\Lambda_G$, we have $\mu_\sigma=n/2$, so the formula in \cref{thm:cycle} for the size of the orbit of $\Theta$ containing $(\sigma,1,1)$ simplifies to $p_\sigma n(n-1)$. If $n=4$, then $p_\sigma=1$ for all $\sigma$, so the order of $\Theta$ is $12$. If $n>4$, then it is straightforward to find labelings $\sigma$ such that $p_\sigma={n-1}$, so the order of $\Theta$ is $n(n-1)^2$. To prove \cref{cor:sieving}, we must show that the number of elements of $\Xi_G$ fixed by $\Theta^{n(n-1)k}$ is $2n^2(n-1)\sum_{\lambda\vdash n-1}f^\lambda_{n-1\mid\maj}f^\lambda(e^{2\pi ik/(n-1)})$. According to \eqref{eq:bold_omega}, it suffices to prove that the number of labelings ${\sigma\in\Lambda_G}$ such that $\Theta^{kn(n-1)}(\sigma,1,1)=(\sigma,1,1)$ is $n(n-1)\sum_{\lambda\vdash n-1}f^\lambda_{n-1\mid\maj}f^\lambda(e^{2\pi ik/(n-1)})$. 

The number of labelings $\sigma\in\Lambda_G$ satisfying $\Theta^{kn(n-1)}(\sigma,1,1)=(\sigma,1,1)$ is $n|\Gamma|$, where \[\Gamma=\{\sigma\in\Lambda_G:\sigma(v_n)=1\text{ and }p_\sigma\text{ divides }k\}.\] Upon inspecting the definition of $p_\sigma$, we find that 
\begin{equation}\label{eq:Gamma_after_Vic}
|\Gamma|=|\{\xi\in \mathfrak S_{n-1}:c^k\xi c^j=\xi\text{ for some integer }j\in[0,n-2]\}|,
\end{equation} 
where $c$ is the cycle $(1\,\, 2\,\,\cdots\,\, n-1)$ in $\mathfrak S_{n-1}$. Barcelo, Reiner, and Stanton proved (see \cite[Theorem~1.4]{BRS}) that for \emph{fixed} $j$, the number of $\xi\in \mathfrak S_{n-1}$ satisfying $c^k\xi c^j$ is \[\sum_{\lambda\vdash n-1}f^\lambda(e^{2\pi ij/(n-1)})f^\lambda(e^{2\pi ik/(n-1)})\] (their result is actually much more general). Therefore, we can use \eqref{eq:plug_in_roots} to find that
\begin{align*}
n|\Gamma|&=
n\sum_{j=0}^{n-2}\sum_{\lambda\vdash n-1}f^\lambda(e^{2\pi ij/(n-1)})f^\lambda(e^{2\pi ik/(n-1)}) \\ 
&= n\sum_{\lambda\vdash n-1}f^\lambda(e^{2\pi ik/(n-1)})\sum_{j=0}^{n-2}f^\lambda(e^{2\pi ij/(n-1)}) \\ 
&= n(n-1)\sum_{\lambda\vdash n-1}f^\lambda_{n-1\mid\maj}f^\lambda(e^{2\pi ik/(n-1)}), 
\end{align*}
as desired. 

\begin{remark}
Vic Reiner has suggested to us the following alternative approach to proving that the quantity on the right-hand side of \eqref{eq:Gamma_after_Vic} is equal to $(n-1)\sum_{\lambda\vdash n-1}f^\lambda_{n-1\mid\maj}f^\lambda(e^{2\pi i k/(n-1)})$. Given a graded ring $U$, let $\mathrm{Hilb}(U,q)$ denote the Hilbert series of $U$ in the variable $q$. Let $\mathfrak S_{n-1}$ act on the polynomial ring $R=\mathbb C[x_1,\ldots,x_{n-1}]$ by permuting variables. For a subgroup $W$ of $\mathfrak S_{n-1}$, let $R^W$ denote the $W$-invariant subalgebra of $R$, and let $(R^{W}_+)$ be the ideal of $R$ generated by elements of $R^W$ of positive degree. Using Springer's theorem on regular elements \cite[Proposition~4.5]{Springer}, one can show that 
\[\sum_{\lambda\vdash n-1}f^{\lambda}_{n-1\mid \maj}f^\lambda(q)=\mathrm{Hilb}( R^{\mathcal C_{n-1}},q)/\mathrm{Hilb}(R^{\mathfrak S_{n-1}},q)=\mathrm{Hilb}(R^{\mathcal C_{n-1}}/(R^{\mathfrak S_{n-1}}_+),q).\] 
The desired result then follows from \cite[Theorem~8.2]{CSP}. 
\end{remark}

\section{Other Directions}
\label{sec:conclusion} 

In this section, we collect several suggestions for future work. 

\subsection{Other Graphs} 
The most natural open problem is to understand the orbit structure of toric promotion with reflections and refractions for other choices of the graph $G=(V,E)$ and the partition $E=\Eflect\sqcup\Efract$. Since \cref{thm:cycle} handled the case in which $G$ is a cycle and $|\Efract|$ is even, one particularly natural setting to consider is that in which $G$ is a cycle and $|\Efract|$ is odd. 

\subsection{Contractible Billiards Trajectories} 
As mentioned in \cref{subsec:remarks}, one can view the combinatorial billiards trajectories determined by the map $\Theta$ as closed loops in the torus $\mathbb T_{n-1}$. It would be very interesting to answer the following. 

\begin{question}\label{quest:contractible}
For which choices of the graph $G=(V,E)$ and the partition $E=\Eflect\sqcup\Efract$ are all of the billiards trajectories determined by $\Theta$ contractible? For which choices are none of the billiards trajectories determined by $\Theta$ contractible? 
\end{question}

Defant and Liu have made some partial progress toward \cref{quest:contractible} \cite{DefantLiu}.

\subsection{Other Weyl Groups}\label{subsec:Weyl} 
If would be interesting to investigate analogues of our work in other affine Weyl groups. In this setting, it is natural to allow the beam of light to travel in the direction of a coroot vector. One should also designate the hyerplanes in the affine Coxeter arrangement to be windows, mirrors, and metalenses in such a way that parallel hyperplanes are made of the same material; this ensures that one can obtain a well-defined combinatorial dynamical system by projecting to the associated finite Weyl group. Even when working with the affine symmetric group, one could consider variants of our setting by choosing to shine the beam of light in the direction of a different coroot vector.

\subsection{Combinatorial Refraction Billiards in a Polygon} 

The article \cite{DefantJiradilok} considers combinatorial billiards systems that are confined to a polygon in the triangular grid. It could be interesting to introduce refractions into such systems. 

\subsection{Directed Graphs and Hyperplanes of Two Materials}

The \emph{Bender--Knuth billiards systems} in \cite{BDHKL} are defined using windows and \emph{one-way mirrors}. A \emph{one-way mirror} is a hyperplane that allows light to pass directly through it from one direction but reflects light from the other direction. More generally, one could consider combinatorial billiards systems in which each hyperplane is made of two materials (which could be the same material). That is, the light will pass directly through, reflect, or refract depending on the side of the hyperplane that it hits. 

One could also consider toric combinatorial billiards systems in which each toric hyperplane in made of two materials. In our toric setting, one could replace $G$ with a directed graph and designate each directed edge as either a reflection edge or a refraction edge. In this case, we would modify the definition of $\Theta(\sigma,i,\epsilon)$ in \eqref{eq:Theta} so that the three cases depend on the ordered pair $(\sigma^{-1}(i),\sigma^{-1}(i+1))$ instead of the unordered pair $\{\sigma^{-1}(i),\sigma^{-1}(i+1)\}$. When $G$ is a tournament on $[n]$ whose directed edges are all reflection edges, this setup is closely related to a version of the Totally Asymmetric Simple Exclusion Process (TASEP) on a cycle (see \cite{Lam, LamWilliams}). 

\section*{Acknowledgments}
Ashleigh Adams was supported by NSF grant DMS-2247089. Colin Defant was supported by the National Science Foundation under Award No.\ 2201907 and by a Benjamin Peirce Fellowship at Harvard University. Jessica Striker was supported by a Simons Foundation gift MP-TSM-00002802 and NSF grant DMS-2247089. We thank Vic Reiner and Ilaria Seidel for very helpful conversations. We thank Pavel Galashin for suggesting the problem of considering the homology of toric combinatorial billiards trajectories. 

\section*{Data Availability Statement} 
This article does not have any associated data. 

\section*{Conflict of Interest Statement} 
The authors have no conflicts of interest to report.

\end{document}